\documentclass[12pt]{amsart}
\usepackage{amsmath,amsfonts,amssymb,amscd,amsthm,amsbsy,upref,epsf, color}
\usepackage[bookmarks]{hyperref}

\newtheorem{theo}{Theorem}

\newtheorem{lemm}[theo]{Lemma}
\newtheorem{rem}[theo]{Remark}

\numberwithin{equation}{section}
\numberwithin{theo}{section}

\newcommand{\R}{\mathbb R}

\def\supp{\operatorname{supp}}

\DeclareMathOperator{\Vol}{Vol}

\newcommand{\Div}{\mathrm{div}}

\newcommand{\dd} {\mathrm{d}}

\numberwithin{equation}{section}

\newcommand{\norm}[1]{\left\Vert#1\right\Vert}
\newcommand{\abs}[1]{\left\vert#1\right\vert}

\newcommand{\F}{\mathbb F}

\DeclareMathOperator{\dv}{div}
\DeclareMathOperator{\Def}{Def}
\DeclareMathOperator{\Ric}{Ric}

\def\H{\mathbb H^{2}(-a^{2})}
\def\be{\begin{equation}}
\def\ee{\end{equation}}

\newcommand{\tV}{\textbf V}

\begin{document}

\title{Antithesis of the Stokes paradox on the hyperbolic plane}

\author[Chan]{Chi Hin Chan}
\address{Department of Applied Mathematics, National Chiao Tung University,1001 Ta Hsueh Road, Hsinchu, Taiwan 30010, ROC}
\email{cchan@math.nctu.edu.tw}

\author[Czubak]{Magdalena Czubak}
\address{Department of Mathematics\\
University of Colorado Boulder\\ Campus Box 395, Boulder, CO, 80309, USA}
\email{czubak@math.colorado.edu}

\begin{abstract}
We show there exists a nontrivial $H^1_0$ solution to the steady Stokes equation on the 2D exterior domain in the hyperbolic plane.  Hence we show there is no Stokes paradox in the hyperbolic setting.  We also show the existence of a nontrivial solution to the steady Navier-Stokes equation in the same setting, whereas the analogous problem is open in the Euclidean case.
\end{abstract}

\maketitle


\section{Introduction}
The Stokes equation is the linear version of the Navier-Stokes.  George Gabriel Stokes proposed the equation as a model for the fluid flow in the low Reynolds regime. 
However, in a situation when we have a long cylinder moving slowly through an incompressible, viscous fluid, in the direction perpendicular to the axis of the cylinder, it follows from the work of Stokes \cite{Stokes1851} that there is no nontrivial solution.  This obviously contradicts what one physically expects. Oseen offered a correction to the Stokes' model, which is now known as the Oseen's system \cite{Oseen1910}.  This lack of the nontrivial solution is called the Stokes paradox. We give a precise description of the problem below.

The Stokes equation on $\R^n$ is
\be\label{StokesRn}
\begin{split}
-\mu\Delta v + \dd p&=0,\\
\dv v&=0,
\end{split}
\ee
where $v$ is the velocity of the fluid, $p$ is the pressure, and the fluid has viscosity $\mu>0$.  If we want to introduce a flow past an obstacle $K \subset R^n$ with an impermeable boundary, we can think of $K$ as a compact set, and then, we append a boundary condition for $v$ on $\partial K$.  Similarly, we can think of $K$ as a rigid body moving through a fluid with a nonzero viscosity $\mu$.  

The motion is assumed to be slow in the sense that the Reynolds number is small,
and motivates the disregard for the nonlinear interactions.  We now think of the motion of the liquid in a frame centered at $K$. 
Let $D_0(1)$ denote the closed unit disk in $\R^2$ and $S^1$ be its boundary.

 If $K$ is a very long cylinder and the motion is a translation that is perpendicular to the axis of the cylinder, then the motion is planar, and one can study
\be\label{StokesExt}
\begin{split}
-\mu\Delta v + \dd p&=0,\\
\dv v&=0,\\
v\big |_{S^1}&=v_0,
\end{split}
\ee
where $v: \R^2-D_0(1) \rightarrow \R^n$,  $p: \R^2-D_0(1) \rightarrow \R$, and $v_0$ is a constant vector in $\R^2$. We also include the condition at infinity to be $v\rightarrow 0$.  This description of the problem can be found in \cite[Chapter 5]{Galdi}.

We now move to the hyperbolic plane, and set up the problem in the analogous way except that instead of prescribing the boundary condition on the obstacle to be a possibly nonzero vector $v_0$, we set it to be zero, which perhaps is even more surprising that this and $v\rightarrow 0$, will lead to a nontrivial solution.  

Let $O$ be a chosen base point on the hyperbolic plane, and  let $\rho(O,x)$ denote the distance function from $O$ to $x$.  For  $R_0>0$, define
\[
\Omega (R_0) = \{ x \in \mathbb{H}^2(-a^2) : \rho (O, x ) > R_0  \} .
\]

We study the following steady Stokes equation on $\Omega (R_0)$, 
\begin{equation}\label{StatStokes}
\begin{split}
-2 \dv \Def v + \dd p & = 0 , \\
\dd^* v & = 0,\\
v\big|_{\partial B_{O}(R_0)}&=0,
\end{split}
\end{equation}
where we set the viscosity $\mu$ to $1$, $p$ is a smooth function on $\Omega (R_0)$, and $\Def$ is the deformation tensor, which can be written in coordinates as
\[
(\Def v)_{ij}=\frac 12(\nabla_i v_j+\nabla_j v_i).
\]
A computation using Ricci identity shows that on the hyperbolic plane the following holds for the divergence free vector fields $v$.
\be\label{viscop}
-2\dv\Def v=-\Delta v -2\Ric v=-\Delta v + 2a^2 v,
\ee
where $-\Delta$ is the Hodge Laplacian.  We use \eqref{viscop} as we believe this is the correct viscosity operator as pointed out in \cite{EbinMarsden}.  For more about the possible forms of the equations, see \cite{CCD16}.
\begin{rem}
We note that the equation $-\Delta v + v=\nabla p$ in the Euclidean setting, just like $-\Delta v=\nabla p$, leads to only trivial solutions \cite{Heywood}.
\end{rem}

In the second part of the paper we are interested in the steady Navier-Stokes equation  

\begin{equation}\label{StatNS}
\begin{split}
-2 \dv \Def v + \nabla_v v+\dd p & = 0 , \\
\dd^* v & = 0,\\
v\big|_{\partial B_{O}(R_0)}&=0.
\end{split}
\end{equation}

We now state our two main results
\begin{theo}\label{main}
There exist a nontrivial solution  $(u, p)$ on $\Omega (R_0)$ to
\begin{equation}\label{StokesEquationi}
\begin{split}
-\triangle u + 2a^2 u + \dd p & = 0 ,\\
\dd^* u & = 0,\\
u \big |_{\partial \Omega (R_0)} & = 0 ,
\end{split}
\end{equation}
where $u$ satisfies
\[
\int_{\Omega(R_0)}\abs{\nabla u}^2 <\infty,
\]
and $p\in L^2_{loc}(\Omega(R_0))$.
\end{theo}
The theorem for the steady Navier-Stokes equation is analagous.
\begin{theo}\label{main2}
There exist a nontrivial solution  $(u, p)$ on $\Omega (R_0)$ to
\begin{equation}\label{SteadyNS}
\begin{split}
-\triangle u + 2a^2 u +\nabla_v v+ \dd p & = 0 ,\\
\dd^* u & = 0,\\
u \big |_{\partial \Omega (R_0)} & = 0 ,
\end{split}
\end{equation}
where $u$ satisfies
\[
\int_{\Omega(R_0)}\abs{\nabla u}^2 <\infty,
\]
and $p\in L^2_{loc}(\Omega(R_0))$.
\end{theo}
Moreover, not only the solution is nontrivial, it is not a potential flow.
\begin{theo}\label{main3}
Let $u$ be a solution obtained in Theorem \ref{main} or \ref{main2}, then
\[
u\neq \dd F,
\]
where $F$ is harmonic.
\end{theo}

 \subsection{Outline of the proof}\label{outline}
 
 Some of the main tools that we use are based on the Euclidean theory, and can be found for example in \cite{Seregin} or \cite{Galdi}.  We give references to these sources when we use them.  Poincar\'e model and the conformal equivalence is what allows us to benefit from these results.  The Poincar\'e model is set up in Section \ref{notation}.  Besides the Euclidean theory, we also heavily rely on our work in \cite{CC13, CC15}.  For example, we use that $L^2$ harmonic 1-forms are actually in $H^1$, as well as the Ladyzhenskaya's inequality in the hyperbolic setting \cite{CC13}, and a Poincar\'e type estimate \cite{CC15}.  These will be referenced or stated when they are needed.
 
The form of the solution for the steady Navier-Stokes equation is inspired by the solution obtained for the steady Stokes equation \eqref{StokesEquationi}.  We now explain the idea behind obtaining a nontrivial solution to \eqref{StokesEquationi}.

The solution $u$ is a sum of three terms
\be\label{defnu}
u = \big (  \eta_{R_0} - 1  \big ) \dd F + w + \widetilde{w} .
\ee
It will be obtained in four steps in Section \ref{sectionONE}.  We will start with the equation that is like \eqref{StokesEquationi} except for the boundary condition on the obstacle.  More precisely, 

\begin{equation}\label{step1}
\begin{split}
-\triangle  v + 2a^2 v + \dd P & = 0 ,\\
\dd^* v & = 0 , \\
v \big |_{\partial \Omega (R_0)} & = \dd F |_{\partial \Omega (R_0)},
\end{split}
\end{equation}
where $\dd F$ is a harmonic $L^2$ form on $\H$, whose existence is guaranteed for example by \cite{Dodziuk}.

One can then see $(v, p)=(\dd F, -2a^2F)$ is a solution to \eqref{step1}.  However, we want to find a solution that is different from $\dd F$, so we let
\[
v=\eta_{R_0}dF+\tilde v, \]
where $\eta_{R_0}$ is a cut-off function, whose properties will be given in Section  \ref{sectionONE}, but in particular $\eta_{R_0}\equiv 1$ on a ball of radius $2R_0$.  Then in order for $v$ to solve \eqref{step1}, one can see $\tilde v$ must solve 
\begin{equation}\label{step2}
\begin{split}
-\triangle \widetilde{v} + 2a^2 \widetilde{v} + \dd P & = \triangle \big ( \eta_{R_0} \dd F  \big )
-2a^2 \eta_{R_0} \dd F ,\\
\dd^* \widetilde{v} & = g ( \dd \eta_{R_0} , \dd F ),\\
\widetilde{v} \big |_{\partial \Omega (R_0)} & = 0.
\end{split}
\end{equation}

Then, in Lemma \ref{KeyLEMMA}, we find $ w$ such that $ w$ satisfies the divergence condition in \eqref{step2}, so
\[
\dd^\ast  w= g ( \dd \eta_{R_0} , \dd F ).
\]

Hence to find $\tilde v$ we set 
\be\label{tildevdef}
\tilde v=\tilde w + w,
\ee
which means $\tilde w$ needs to solve
\begin{equation}\label{step3}
\begin{split}
-\triangle  \widetilde{w} + 2a^2 \widetilde{w} + \dd P & = \triangle\big ( \eta_{R_0} \dd F + w  \big ) -2a^2 \big ( \eta_{R_0} \dd F + w \big )  ,\\
\dd^* \widetilde{w} & = 0 ,\\
\widetilde{w} \big |_{\partial \Omega (R_0)} & = 0 .
\end{split}
\end{equation}
The system \eqref{step3} is then solved by the Riesz Representation Theorem,  in Lemma \ref{TrivialLemma}.  Since we already had $w$, from \eqref{tildevdef} we now have $\tilde v$, so we then backtrack to get $v$, and then finally, because we want a solution that has zero boundary data, we define $u$ by \eqref{defnu}.

To finish we will need to recover the pressure, and show $u$ is nontrivial (Steps 4 and 5 in Section \ref{sstep4} and Section \ref{sstep5} respectively).  In Section \ref{step6}, Step 6, we show $u$ is not a potential flow.
Section \ref{sectionTWO} is devoted to constructing a solution to the steady Navier-Stokes equation.  Here we use Leray's original approach of constructing the solution in the exterior domain as described in \cite{Seregin}.

\subsection{Acknowledgments}
The authors would like to thank
Vladim{\'i}r  \v{S}ver{\'a}k and Hideo Kozono for suggesting to consider the Stokes paradox in the hyperbolic setting.  

C. H. Chan is partially supported by a
grant from the National Science Council of Taiwan (NSC 105-2918-I-009 -008).
M. Czubak is partially supported by a grant from the Simons Foundation \# 246255.

\section{Preliminaries}\label{notation}

The geodesic ball at $x$ with radius $R$ in $\Bbb H^{2}(-a^{2})$ will be denoted by $$B_{x}(R)=\{y\in \mathbb H^{2}(-a^{2}):\rho(x,y)< R\},$$
where $\rho(x,y)$ is the geodesic distance between $x$ and $y$ in $\Bbb{H}^{2}(-a^{2}).$ For any $x\in \Bbb R^{2}$ and $R>0$, the Euclidean open ball centered at $x$ with radius $R$ will be denoted by $$D_{x}(R)=\{y\in \Bbb R^{2}:|x-y|<R\}.$$

\subsection{Poincar\'e disk model}
We recall the standard Poincar\'e disk model $Y: \mathbb{H}^2(-a^2) \rightarrow D_0(1)$, under which the hyperbolic metric on $\mathbb{H}^2(-a^2)$ can be represented by
\begin{equation*}
g(\cdot , \cdot ) = \frac{4}{a^2 \big ( 1-|Y|^2\big )^2} \Big \{ \dd Y^1 \otimes \dd Y^1 + \dd Y^2 \otimes \dd Y^2 \Big \}.
\end{equation*}

If we let $\tilde y \in D_0(1)$ with $\abs{\tilde y}=r$, then by parametrizing the straight line connecting $0$ and $\tilde y$, the geodesic distance between $0$ and $\tilde y$ is
 \be\label{gdist}
\rho(0,\tilde y)=\frac 1a\int^r_0 \frac{2}{1-t^2} d t =\frac 1a \log(\frac{1+r}{1-r}).
 \ee
A computation shows that if
\[
r=\tanh(\frac{aR}{2}),
\]
then
\[
\frac 1a \log(\frac{1+r}{1-r})=R.
\]
This means that $Y$ maps a geodesic ball of radius $R$ onto the Euclidean ball of radius $\tanh(\frac{aR}{2})$.

Next, for some fixed $R_0 > 0$,  let $\Omega (R_0) = \{ x \in \mathbb{H}^2(-a^2) : \rho (O, x ) > R_0  \} $, so under $Y$, this corresponds to
\be\label{eucl_a}
\{ y \in \mathbb{R}^2 : \tanh (\frac{aR_0}{2}) < |y| < 1  \}.
\ee

\subsection{Relating $H^1_0$ norms}
We now discuss the relationship of the $H^1_0$ norm, on the set $U\subset \H$, to the Euclidean $H^1_0$ norm, on the set $Y(U)\subset D_0(1)$.  More precisely, if $u$ is a $1-$form on $U\subset \H$ written in coordinates as
\[
u = u_1 \dd Y^1 + u_2 \dd Y^2 ,
\]
then we can define $u^\sharp$ on $Y(U)$ by
\begin{equation*}
u^\sharp = {u}_1^\sharp \dd y^1 +{u}_2^\sharp \dd y^2 ,
\end{equation*}
where
\begin{equation*}
\begin{split}
{u}_1^\sharp = {u}_{1}\circ Y^{-1} , \quad
{u}_2^\sharp  = {u}_{2}\circ Y^{-1} .
\end{split}
\end{equation*}
In other words, $u^\sharp$ is the pull-back of $u$ by $Y^{-1}$
\[
u^\sharp=\big ( Y^{-1}\big )^* u.
\]

We show
\be\label{easy_inclusion}
\norm{u}_{H^1_0(U)}\lesssim \norm{u^\sharp} _{H^1_0(Y(U))}.
\ee
First
\begin{align*}
\norm{u}^2_{L^2(U)}&=\int_U g(u,u) \Vol_{\H}\\
&=\int_{U}\frac{a^2(1-\abs{Y}^2)^2}{4}((u_1)^2+(u_2)^2) \frac{4}{a^2(1-\abs{Y}^2)^2}\dd Y^1 \wedge \dd Y^2\\
&=\int_{Y(U)}(u_1^\sharp)^2+(u_2^\sharp)^2 \dd y^1 \dd y^2\\
&=\norm{u^\sharp}^2_{L^2(Y(U))}.
\end{align*}
Next we estimate the $L^2$ norm of the covariant derivative of $u$.  A computation shows (see for example \cite[Appendix A.3]{CC13})
\begin{equation}
\begin{split}
{\nabla} u & = (\nabla u)_{ij} \dd Y^i \otimes \dd Y^j\\
&=\big \{ \frac{\partial u_1}{\partial Y^1} -\frac{2Y^1u_1}{1-|Y|^2} + \frac{2Y^2u_2}{1-|Y|^2}   \big \} \dd Y^1\otimes \dd Y^1 \\
& + \big \{\frac{\partial u_2}{\partial Y^1} -\frac{2Y^2u_1}{1-|Y|^2} - \frac{2Y^1u_2}{1-|Y|^2}   \big \} \dd Y^1\otimes \dd Y^2 \\
& + \big \{ \frac{\partial u_1}{\partial Y^2} -\frac{2Y^2u_1}{1-|Y|^2} - \frac{2Y^1u_2}{1-|Y|^2}   \big \} \dd Y^2\otimes \dd Y^1 \\
& + \big \{ \frac{\partial u_2}{\partial Y^2} +\frac{2Y^1u_1}{1-|Y|^2} - \frac{2Y^2u_2}{1-|Y|^2}   \big \} \dd Y^2\otimes \dd Y^2.
\end{split}
\end{equation}
Moreover
\[
g(\nabla u, \nabla u)=\left(\frac{a^2(1-\abs{Y}^2)^2}{4}\right)^2 (\nabla u)_{ij}(\nabla u)_{ij}.
\]
Now for each $(i, j), \ 1\leq i, j \leq 2$, we have  $(\nabla u)_{ij}^2$ is bounded by a term that looks like
\[
(\frac {\partial u_j}{\partial Y^i} )^2 + \frac{4\abs{Y}^2(u_1^2+u_2^2)}{(1-\abs{Y}^2)^2},
\]
and by definition
\[
\frac {\partial u_k}{\partial Y^l}\Big |_p=\frac{\partial u_k^\sharp}{\partial y^l}\Big |_{Y(p)}.
\]
It follows
\begin{align}
\norm{\nabla u}^2_{L^2(U)}&=\int_U\left(\frac{a^2(1-\abs{Y}^2)^2}{4}\right)^2 (\nabla u)_{ij}(\nabla u)_{ij} \frac{4}{a^2(1-\abs{Y}^2)^2}\dd Y^1 \wedge \dd Y^2\nonumber\\
&\lesssim \int_{Y(U)}\frac{a^2(1-\abs{y}^2)^2}{4}\left(\abs{\nabla u^\sharp}^2+ \frac{4\abs{y}^2\abs{u^\sharp}^2}{(1-\abs{y}^2)^2}\right)\dd y^1 \dd y^2\nonumber\\
&\lesssim a^2\int_{Y(U)} \left(\abs{\nabla u^\sharp}^2+ \abs{u^\sharp}^2\right)\dd y^1 \dd y^2\nonumber\\
&=a^2\norm{u^\sharp}^2_{H^1_0(Y(U))},\label{easy_inc2}
\end{align}
 where we have used $\abs{y}\leq 1$ on $Y(U)$.

\section{Proof of Theorem \ref{main}: Existence of a nontrivial solution to the Stokes equation on $\Omega(R_0)$.}\label{sectionONE}

As explained in Section \ref{outline}, the solution is obtained in several steps.  
\subsection{Step 1: Setup.}\label{sstep1}
To begin the discussion, we define the function space
\begin{equation}
\mathbb{F} = \big \{ \dd F \in L^2(\mathbb{H}^2(-a^2)) : -\triangle  F = 0   \big \}.
\end{equation}
One can show $\F \subset H^1_0(\H)$ \cite{CC13}.
Consider  the following system on the exterior domain $\Omega (R_0)$, where a nontrivial $\dd F \in \mathbb{F}$ is given.
\begin{equation}\label{HelpfulNSONE}
\begin{split}
-\triangle  v + 2a^2 v + \dd P & = 0 ,\\
\dd^* v & = 0 , \\
v \big |_{\partial \Omega (R_0)} & = \dd F |_{\partial \Omega (R_0)}.
\end{split}
\end{equation}

It is clear that $(v, P)=(\dd F, -2a^2F)$ is a solution to equation \eqref{HelpfulNSONE}. But we are interested in seeking nontrivial solution $v$ to the system \eqref{HelpfulNSONE}, which differs from $\dd F$. To this end, we choose a fixed cut off function $\eta \in C^{\infty}([0, \infty ))$, which satisfies
\begin{equation}\label{cutoffcondition}
\begin{split}
 \chi_{[0,1]} \leq \eta \leq \chi_{[0,2)}, \quad \eta' \leq 0 .
\end{split}
\end{equation}
By means of $\eta$, we can now consider a radially symmetric cut-off function $\eta_{R_0 } \in C_c^{\infty} (B_O(4R_0))$  defined by
\begin{equation}
\eta_{R_0} (x) = \eta\big (\frac{\rho(x)}{2R_0} \big ).
\end{equation}
Then $\eta_{R_0}$  satisfies
\begin{equation}\label{cutoffconditiontwo}
\chi_{B(2R_0)} \leq \eta_{R_0} \leq \chi_{B(4R_0)}.
\end{equation}
We now look for a nontrivial solution $v$ to \eqref{HelpfulNSONE} in the form of
\begin{equation}\label{VandTildeV}
v = \eta_{R_0} \dd F + \widetilde{v} ,
\end{equation}
where $\widetilde{v} \in H^1_0(\Omega (R_0 ))$.  

Since $v$, as a solution to \eqref{HelpfulNSONE}, must be divergence free on $\Omega (R_0)$, by \eqref{VandTildeV}, and using $\dd^\ast (fu)=-g(\dd f, u)+ f \dd^\ast u$, for a function $f$ and a 1-form $u$, it follows that we need
\begin{equation}\label{divergencetildeV}
\dd^* \widetilde{v} = g (\dd \eta_{R_0} , \dd F ).
\end{equation}
So we observe that in order to find a solution $v$ to \eqref{HelpfulNSONE} on $\Omega (R_0)$, which is in the form of \eqref{VandTildeV}, we have to demonstrate how to find a solution $\widetilde{v}\in H^1_0(\Omega (R_0))$ to 
 
\begin{equation}\label{HelpfulNSTildev}
\begin{split}
-\triangle \widetilde{v} + 2a^2 \widetilde{v} + \dd P & = \triangle \big ( \eta_{R_0} \dd F  \big )
-2a^2 \eta_{R_0} \dd F ,\\
\dd^* \widetilde{v} & = g ( \dd \eta_{R_0} , \dd F ),\\
\widetilde{v} \big |_{\partial \Omega (R_0)} & = 0.
\end{split}
\end{equation}

\subsection{Step 2: Finding $w$ such that $\dd^* w  = g ( \dd \eta_{R_0} , \dd F) $ } \label{sstep2} To find $\tilde v$, which solves \eqref{HelpfulNSTildev}, we first find a function $w$ that satisfies the divergence condition in \eqref{HelpfulNSTildev}.  We start by making some remarks.

Under the coordinate system $Y$, introduced in Section \ref{notation}, any $1$-form $u$ on $\Omega (R_0)$ can be written as
\begin{equation*}
u = u_1 \dd Y^1 + u_2 \dd Y^2 ,
\end{equation*}
with the divergence expressed as
\begin{equation}\label{hypdiv}
- \dd ^* u = \frac{a^2 \big ( 1-|Y|^2\big )^2 }{4} \Big ( \frac{\partial u_1}{\partial Y^1} + \frac{\partial u_2}{\partial Y^2}   \Big ).
\end{equation}
We also have
\begin{equation}\label{hypextra}
g(\dd \eta_{R_0} , \dd F ) = \frac{a^2 \big ( 1-|Y|^2\big )^2 }{4} \Big (  \frac{\partial \eta_{R_0} }{\partial Y^1 }\frac{\partial F }{\partial Y^1 } + \frac{\partial \eta_{R_0} }{\partial Y^2 } \frac{\partial F }{\partial Y^2}    \Big ).
\end{equation}
So, in light of \eqref{hypdiv} and \eqref{hypextra},  \eqref{divergencetildeV} is equivalent to
\begin{equation}\label{Easytoregconize}
\Big ( \frac{\partial \widetilde{v}_1 }{\partial Y^1 } + \frac{\partial \widetilde{v}_2 }{\partial Y^2 } \Big ) = -\Big (  \frac{\partial \eta_{R_0} }{\partial Y^1 }\frac{\partial F }{\partial Y^1 } + \frac{\partial \eta_{R_0} }{\partial Y^2 } \frac{\partial F }{\partial Y^2}    \Big ).
\end{equation}

Next, we consider the following two functions\begin{equation*}
\begin{split}
\widetilde{v}_1^\sharp & = \widetilde{v}_{1}\circ Y^{-1} , \\
\widetilde{v}_2^\sharp & = \widetilde{v}_{2}\circ Y^{-1} .
\end{split}
\end{equation*}
Note, in view of the discussion leading to \eqref{eucl_a}, since we study the problem on the exterior domain $\Omega(R_0)\subset \H$, $\widetilde{v}_1^\sharp$ and $\widetilde{v}_2^\sharp$
are defined in $\{ y \in \mathbb{R}^2 : \tanh (\frac{aR_0}{2}) < |y| < 1  \}$ .

We will also consider the following two smooth functions on $D_0(1)$.
\begin{equation*}
\begin{split}
\eta^{\sharp} & = \eta_{R_0} \circ Y^{-1} , \\
F^{\sharp}  & = F \circ Y^{-1} .
\end{split}
\end{equation*}
With the above notation,  \eqref{Easytoregconize} is equivalent to
\begin{equation}\label{EuclidDivcondition}
\Div_{\mathbb{R}^2} \widetilde{v}^{\sharp} = - \nabla_{\mathbb{R}^2} \eta^{\sharp} \cdot \nabla_{\mathbb{R}^2} F^{\sharp},
\end{equation}
which has to be satisfied by $\widetilde{v}^{\sharp}$ on $\{ \tanh{\frac{aR_0}{2}} < |y| <1      \}$.

Note however that
\begin{equation}\label{suppcondition}
\supp \big ( \nabla_{\mathbb{R}^2} \eta^{\sharp} \cdot \nabla_{\mathbb{R}^2} F^{\sharp} \big  )
\subset \Big \{ y \in \mathbb{R}^2 : \tanh \big (aR_0 \big )  \leq |y| \leq \tanh \big ({2a} R_0 \big ) \Big \}.
\end{equation}

We now invoke the following result from the theory of the Navier-Stokes equation in the Euclidean setting.

\begin{theo}\cite[Thm III.3.3]{Galdi}, \cite{Seregin}\label{Classicalresult}
Let $\Omega$ be a bounded domain with smooth boundary in $\mathbb{R}^N$, with $N\geq 2$. Given $h \in L^2(\Omega )$ satisfying
\begin{equation}\label{crucialcondition}
\int_{\Omega} h \Vol_{\mathbb{R}^N} = 0,
\end{equation}
there exists at least one $U \in H^1_0(\Omega )$ such that
\begin{equation*}
\Div_{\mathbb{R}^N} U = h
\end{equation*}
holds in the weak sense on $\Omega$, and that the following a priori estimate holds
\begin{equation*}
\big \| \nabla_{\mathbb{R}^N} U  \big \|_{L^2(\Omega )} \leq C(N, \Omega ) \big \| h \big \|_{L^2(\Omega )} .
\end{equation*}
\end{theo}

With the help of Theorem \ref{Classicalresult}, we establish the following lemma.
\begin{lemm}\label{KeyLEMMA}
There exists some $1$-form $w \in H^1_0( B_O(4R_0) - \overline{B_O(2R_0)}) $ such that
\begin{equation}\label{helpfulhypdivergence}
\dd^* w = g ( \dd \eta_{R_0} , \dd F   )
\end{equation}
holds weakly on $\mathbb{H}^2(-a^2)$, and
 \begin{equation}\label{H1estimateforW}
\big \| \nabla w \big \|_{L^2(\mathbb{H}^2(-a^2))} \leq C(a, R_0) \big \| \dd F \big  \|_{L^2(\mathbb{H}^2(-a^2))} ,
\end{equation}
where $C(a, R_0)$ is some positive absolute constant depending only on $a$ and $R_0$.
\end{lemm}

\begin{proof}
Let
\begin{equation*}
A(R_0) = \Big \{  y \in \mathbb{R}^2 : \tanh \big ( aR_0 \big ) < |y| < \tanh \big ( 2a R_0 \big )              \Big \}.
\end{equation*}

By \eqref{suppcondition},  the support of $\nabla_{\mathbb{R}^2}\eta^{\sharp} \cdot \nabla_{\mathbb{R}^2}F^{\sharp}$ is included in $\overline{A(R_0)}$. We also note that
\begin{equation*}
Y^{-1} \big ( A(R_0 )   \big ) = \Big \{ x \in \mathbb{H}^2(-a^2) :  2R_0 < \rho (x) < 4 R_0 \Big \} .
\end{equation*}
In order to apply Theorem \ref{Classicalresult}, we have to check if $h = - \nabla_{\mathbb{R}^2}\eta^{\sharp} \cdot \nabla_{\mathbb{R}^2}F^{\sharp}$ verifies \eqref{crucialcondition}.
So we carry out the following computation.

\begin{equation}\label{standard}
\begin{split}
& \int_{A(R_0)} - \nabla_{\mathbb{R}^2}\eta^{\sharp} \cdot \nabla_{\mathbb{R}^2}F^{\sharp} \Vol_{\mathbb{R}^2} \\
= & \int_{A(R_0)} - \Div \Big \{  \eta^{\sharp} \nabla_{\mathbb{R}^2} F^{\sharp}    \Big \} \Vol_{\mathbb{R}^2} \\
= & - \int_{\big \{ |y| = \tanh (2aR_0) \big \} }  \eta^{\sharp} \nabla_{\mathbb{R}^2} F^{\sharp} \cdot \frac{y}{|y|} \dd S \\
&\qquad+ \int_{\big \{ |y| = \tanh (aR_0) \big \} }  \eta^{\sharp} \nabla_{\mathbb{R}^2} F^{\sharp} \cdot \frac{y}{|y|} \dd S
\end{split}
\end{equation}
Since
\begin{equation*}
\begin{split}
\eta^{\sharp} \big |_{\big \{ |y| = \tanh (2aR_0) \big \}  } & = 0 ,\\
\eta^{\sharp} \big |_{\big \{ |y| = \tanh (aR_0) \big \}  } & = 1,
\end{split}
\end{equation*}
it follows from \eqref{standard} that we have
\begin{equation}\label{easyONE}
\int_{A(R_0)} - \nabla_{\mathbb{R}^2}\eta^{\sharp} \cdot \nabla_{\mathbb{R}^2}F^{\sharp} \Vol_{\mathbb{R}^2}
= \int_{\big \{ |y| = \tanh (aR_0) \big \} }   \nabla_{\mathbb{R}^2} F^{\sharp}\cdot \frac{y}{|y|} \dd S .
\end{equation}
Since $F^{\sharp}$ is harmonic on $D_0(1)$, we have
\begin{equation}\label{EasyTwo}
\begin{split}
0 & = \int_{\big \{ |y| < \tanh ({a} R_0) \big \} } \Div_{\mathbb{R}^2} \Big ( \nabla_{\mathbb{R}^2} F^{\sharp} \Big ) \Vol_{\mathbb{R}^2} \\
& = \int_{\big \{ |y| = \tanh ({a} R_0) \big \}}   \nabla_{\mathbb{R}^2} F^{\sharp}\cdot \frac{y}{|y|} \dd S .
\end{split}
\end{equation}
\eqref{easyONE} and \eqref{EasyTwo} together imply that
\begin{equation}\label{important}
\int_{A(R_0 , a)} - \nabla_{\mathbb{R}^2}\eta^{\sharp} \cdot \nabla_{\mathbb{R}^2}F^{\sharp} \Vol_{\mathbb{R}^2} = 0.
\end{equation}
Because \eqref{important} holds, we can apply Theorem \ref{Classicalresult} to deduce that there exists  at least one $w^{\sharp} = w^{\sharp}_1 \dd y^1 + w^{\sharp}_2 \dd y^2 \in H^1_0(A(R_0))$ such that
\begin{equation}\label{EuclidDivTwo}
\Div_{\mathbb{R}^2} w^{\sharp} = - \nabla_{\mathbb{R}^2} \eta^{\sharp}\cdot \nabla_{\mathbb{R}^2} F^{\sharp} .
\end{equation}

Since $w^{\sharp} \in H^1_0(A(R_0))$, we can think of $w^{\sharp}$ as an element also in $H^1_0(D_0(1))$ through zero extension beyond $A(R_0)$, so 
\begin{equation}\label{inclusion}
w^{\sharp} \in H^1_0(A(R_0 )) \subset H^1_0(D_0(1)).
\end{equation}
It is easy to check that $w^{\sharp}$ once being treated as an element in $H^1_0(D_0(1))$ still satisfies \eqref{EuclidDivTwo} weakly on the whole unit disc $D_0(1)$.
 
Finally, we simply define the $1$-form $w$ by
\begin{equation}\label{definitionofw}
w = Y^* w^{\sharp} = w_1^{\sharp}\circ Y \dd Y^1 + w_2^{\sharp} \circ Y \dd Y^2 .
\end{equation}
Then \eqref{inclusion} and \eqref{easy_inclusion} imply
\begin{equation}\label{HypInclusion}
w \in H^1_0 \big ( B_O(4R_0) - \overline{B_O(2R_0)} \big ) \subset H^1_0(\mathbb{H}^2(-a^2)).
\end{equation}
Moreover, since  \eqref{EuclidDivTwo} holds on $D_0(1)$, it follows that $w$ satisfies \eqref{helpfulhypdivergence} on $\mathbb{H}^2(-a^2)$.
So, the proof of Lemma \ref{KeyLEMMA} is now completed.
\end{proof}

\subsection{Step 3: Further Reduction}\label{sstep3}
Being backed up by Lemma \ref{KeyLEMMA}, we can now try to find some element $\widetilde{v}$ which lies in $H^1_0(\Omega(R_0))$, and which is a solution to the system \eqref{HelpfulNSTildev} on $\Omega(R_0)$. To this end, we define

\begin{equation}\label{Defoftildew}
\widetilde{w} = \widetilde{v} - w ,
\end{equation}
where $\widetilde{v} \in H^1_0(\Omega (R_0))$ is a solution to the system \eqref{HelpfulNSTildev}, and
$w \in H^1_0\big( B_O(4R_0) - \overline{B_O(2R_0)}     \big )$ is some element which satisfies \eqref{helpfulhypdivergence} on $\mathbb{H}^2(-a^2)$. It is then clear that $\widetilde{w}$ lies in $H^1_0(\Omega (R_0))$.

Since $\widetilde{v}$ is supposed to be a solution to \eqref{HelpfulNSTildev} on $\Omega (R_0)$, it follows that $\widetilde{w}$ is a solution to the following system on $\Omega (R_0)$.
\begin{equation}\label{NSaboutWidetildeW}
\begin{split}
-\triangle  \widetilde{w} + 2a^2 \widetilde{w} + \dd P & = \triangle\big ( \eta_{R_0} \dd F + w  \big ) -2a^2 \big ( \eta_{R_0} \dd F + w \big )  ,\\
\dd^* \widetilde{w} & = 0 ,\\
\widetilde{w} \big |_{\partial \Omega (R_0)} & = 0 .
\end{split}
\end{equation}

It is now clear that: in order to find a solution $\widetilde{v}$ to system \eqref{HelpfulNSTildev} on $\Omega (R_0)$, we just have to demonstrate the existence of a solution $\widetilde{w}$ to system \eqref{NSaboutWidetildeW} on $\Omega (R_0)$.  We  do this by writiting down a weak formulation for the system \eqref{NSaboutWidetildeW} on $\Omega (R_0)$.

For this purpose, we look for $\widetilde{w}$ as an element in
\begin{equation}\label{definitionofV}
\textbf{V}(\Omega(R_0)) = \overline{\Lambda_{c,\sigma}^1\big ( \Omega (R_0) \big )}^{\|\cdot \|_{H^1}} .
\end{equation}
Here, the symbol $\Lambda_{c,\sigma}^1\big ( \Omega (R_0) \big )$ stands for the space of all smooth, compactly supported divergence free $1$-forms on $\Omega (R_0)$. Notice that $\textbf{V}(\Omega (R_0))$ is a Hilbert space equipped with the following inner product structure. 
\begin{equation}\label{innerproduct}
\begin{split}
&(( \varphi_1 , \varphi_2   ))_{H^1_0(\Omega (R_0))} =2\int_{\Omega(R_0)}g(\Def \phi_1, \Def \phi_2)\Vol_{\H}\\ 
&\qquad\qquad= \int_{\Omega (R_0)} g (\dd \varphi_1 , \dd \varphi_2 ) \Vol_{\mathbb{H}^2(-a^2)}
+ 2 a^2 \int_{\Omega (R_0)} g (\varphi_1 , \varphi_2 ) \Vol_{\mathbb{H}^2(-a^2)} .
\end{split}
\end{equation}
Here, the dual space of the Hilbert space $\textbf{V}(\Omega (R_0))$ is denoted by  $\textbf{V}'(\Omega (R_0))$.
Now, we use the following abbreviation
\begin{equation}\label{ForcingtermT}
\textbf{T} = \triangle \big ( \eta_{R_0} \dd F + w  \big ) -2a^2 \big ( \eta_{R_0} \dd F + w \big )  .
\end{equation}
Since
\begin{equation*}
\begin{split}
\dd F & \in \mathbb{F}\subset H^1_0(\H) ,\\
w & \in H^1_0\big ( B_O(4R_0) - \overline{B_O(2R_0)} \big ) ,
\end{split}
\end{equation*}
it follows that  
\begin{equation}\label{dualspace}
\textbf{T} \in H^{-1} (\Omega (R_0) ) \subset \textbf{V}'(\Omega (R_0)) .
\end{equation}
So to solve  \eqref{NSaboutWidetildeW} in a weak sense in $ \textbf{V}(\Omega (R_0))$ means to find $\tilde w$ such that
\[
(( \widetilde{w} , \varphi    ))_{H^1_0 (\Omega (R_0))} = \big < \textbf{T} , \varphi     \big >_{ \textbf{V}'(\Omega (R_0)) \otimes \textbf{V}(\Omega (R_0))      } 
\]
holds for all $\varphi \in  \textbf{V}(\Omega (R_0))  $.

With the above preparation, we can now state the following lemma, whose proof is a straightforward consequence of the Riesz Representation Theorem.
\begin{lemm}\label{TrivialLemma}
Consider the forcing term $T$ as specified in \eqref{ForcingtermT}. Then, there exists a uniquely determined element
$\widetilde{w} \in \textbf{V}(\Omega (R_0)) $ such that the following relation holds for any  $1$-form
$\varphi \in  \textbf{V}(\Omega (R_0))  $
\begin{equation}\label{weakformulationONE}
(( \widetilde{w} , \varphi    ))_{H^1_0 (\Omega (R_0))} = \big < \textbf{T} , \varphi     \big >_{ \textbf{V}'(\Omega (R_0)) \otimes \textbf{V}(\Omega (R_0))      } .
\end{equation}
Consequently, such a unique $\widetilde{w} \in \textbf{V}(\Omega (R_0)) $ is also a weak solution to the system \eqref{NSaboutWidetildeW} on $\Omega (R_0)$.
\end{lemm}

\subsection{Step 4: Finding the pressure and putting everything together}\label{sstep4}
Now, we consider such a unique $\widetilde{w} \in \textbf{V}(\Omega (R_0)) $ for which \eqref{weakformulationONE} holds for all $1$-forms $\varphi \in \textbf{V}(\Omega (R_0)) $. Define
$$\textbf{L} : H^1_0 (\Omega (R_0)) \rightarrow \mathbb{R}$$ 
by
\begin{equation}\label{Llinear}
\Big < \textbf{L} , \varphi \Big >_{H^{-1}(\Omega (R_0)) \otimes H^1_0 (\Omega (R_0))} =
(( \widetilde{w} , \varphi  ))_{H^1_0(\Omega (R_0))}  - \Big < \textbf{T}  , \varphi \Big >_{H^{-1} (\Omega (R_0)) \otimes H^1_0 (\Omega (R_0))} .
\end{equation}
The validity of \eqref{weakformulationONE} for all $\varphi \in \textbf{V} (\Omega (R_0))$ simply says that the operator $\textbf{L}$ as given in \eqref{Llinear} satisfies
\begin{equation*}
\textbf{L} \big |_{\textbf{V}(\Omega (R_0))} = 0.
\end{equation*}

This allows us to apply Lemma \ref{RecoverPressure} to $\textbf{L}$ and deduce that there exists a function $P \in L^2_{loc} (\Omega (R_0))$ such that the following relation holds for any $\varphi \in \Lambda_c^1 (\Omega (R_0))$.

\begin{equation*}
(( \widetilde{w} , \varphi  ))_{H^1_0(\Omega (R_0))} + \int_{\Omega (R_0)} P \dd^* \varphi \Vol_{\mathbb{H}^2(-a^2)} = \Big < \textbf{T}  , \varphi \Big >_{H^{-1} (\Omega (R_0)) \otimes H^1_0 (\Omega (R_0))} .
\end{equation*}

This shows that the pair $(\widetilde{w} , P)$ constitutes a weak solution to the system \eqref{NSaboutWidetildeW} on $\Omega (R_0)$ .

So, consequently, 
\begin{equation}\label{tildeVbytildeW}
\widetilde{v} = \widetilde{w} + w \subset  H^1_0(\Omega (R_0))
\end{equation}
and $(\widetilde{v} , P)$ constitutes a weak solution to the system \eqref{HelpfulNSTildev} on $\Omega (R_0)$. Hence, we now take $v$ to be
\begin{equation}\label{FinalformofV}
\begin{split}
v & = \eta_{R_0} \dd F + \widetilde{v} \\
& = \eta_{R_0} \dd F + \widetilde{w} + w.
\end{split}
\end{equation}
Then, it follows that $v \in H^1(\Omega (R_0)) $, and that the pair $\big ( v , P \big )$ is a weak solution to the system \eqref{HelpfulNSONE} on $\Omega (R_0)$ . 

Next, we   consider the $1$-form $u$  given by
\begin{equation}\label{definitionofU}
\begin{split}
u & = v - \dd F  \\
& = \eta_{R_0} \dd F + \widetilde{w} + w - \dd F .
\end{split}
\end{equation}
Since
\begin{equation*}
\begin{split}
& \Big \{ \big ( \eta_{R_0} - 1 \big ) \dd F \Big \} \Big |_{\partial \Omega (R_0))}  = 0 ,\\
& w , \widetilde{w}  \in H^1_0(\Omega (R_0)),
\end{split}
\end{equation*}
it follows that $u$ satisfies
\begin{equation*}
u \in H^1_0(\Omega (R_0)) ,
\end{equation*}
and that the pair $\big ( u , p \big )=\big ( u , P+2a^2 F \big )$ constitutes a weak solution to the steady Stokes system on $\Omega (R_0)$ given in 
 \eqref{StokesEquationi}, which we restate below for convenience. 
 
\begin{equation}\label{StandardNSequation}
\begin{split}
-\triangle u + 2a^2 u + \dd p & = 0 ,\\
\dd^* u & = 0,\\
u \big |_{\partial \Omega (R_0)} & = 0 .
\end{split}
\end{equation}

\subsection{Step 5: Showing the solution is nontrivial}\label{sstep5}
Now, we prove that $u$ is indeed \emph{nontrivial}. To this end, we first establish 
\begin{lemm}\label{judgement}
Consider a radially symmetric cut-off function $\eta_{R_0} \in C_c^{\infty} (\mathbb{H}^2(-a^2))$, which satisfies \eqref{cutoffconditiontwo}, and take a nontrivial harmonic function $F$ on $\mathbb{H}^2(-a^2)$ for which we have $\dd F \in \mathbb{F}$. Let $w \in H^1_0( B_O(4R_0) - \overline{B_O(2R_0)}) $ satisfy \eqref{helpfulhypdivergence} on $\mathbb{H}^2(-a^2)$. Then
\begin{equation}\label{criteriaineq}
\int_{\Omega (R_0)} g(w , \dd F  ) \Vol_{\mathbb{H}^2(-a^2)} = - \int_{\mathbb{H}^2(-a^2)} \eta_{R_0} g(\dd F , \dd F ) \Vol_{\mathbb{H}^2(-a^2)} < 0.
\end{equation}
\end{lemm}
\begin{proof}
Let $w$ be the $1$-form as described in the hypothesis of Lemma \ref{judgement}. Since
\begin{equation*}
w \in H^1_0( B_O(4R_0) - \overline{B_O(2R_0)}) ,
\end{equation*}
it follows that
\begin{equation}\label{Boring4}
\int_{\Omega (R_0)} g (w, \dd F ) \Vol_{\mathbb{H}^2(-a^2)} = \int_{\Omega (R_0)} F \dd^* w  \Vol_{\mathbb{H}^2(-a^2)}.
\end{equation}
But by \eqref{helpfulhypdivergence} 
\begin{equation}\label{boring5}
\int_{\Omega (R_0)} F \dd^* w  \Vol_{\mathbb{H}^2(-a^2)} = \int_{\Omega (R_0)} F  g( \dd \eta_{R_0} , \dd F  ) \Vol_{\mathbb{H}^2(-a^2)}.
\end{equation}
Hence
\begin{equation}\label{boring6}
\begin{split}
& \int_{\Omega (R_0)} g (w, \dd F ) \Vol_{\mathbb{H}^2(-a^2)} \\
= & \int_{\Omega (R_0)} F g( \dd \eta_{R_0} , \dd F  ) \Vol_{\mathbb{H}^2(-a^2)} \\
= & \int_{\Omega(R_0)} F \Big ( \frac{\partial \eta_{R_0}}{\partial Y^1} \frac{\partial F}{\partial Y^1} +
\frac{\partial \eta_{R_0}}{\partial Y^2} \frac{\partial F}{\partial Y^2} \Big ) \dd Y^1 \wedge \dd Y^2 \\
= & \int_{ \big \{ \tanh (\frac{a}{2}R_0) < |y| < 1 \big \} } F^{\sharp}  \nabla_{\mathbb{R}^2}\eta^{\sharp} \cdot  \nabla_{\mathbb{R}^2} F^{\sharp} \dd y^1 \dd y^2 ,
\end{split}
\end{equation}
where  we recall the functions $\eta^{\sharp}$ and $F^{\sharp}$ are given by
\begin{equation*}
\begin{split}
\eta^{\sharp} & = \eta_{R_0} \circ Y^{-1} , \\
F^{\sharp} & = F \circ Y^{-1} .
\end{split}
\end{equation*}
Since $\triangle F = 0$ holds on $\mathbb{H}^2(-a^2)$, we know that $\triangle_{\mathbb{R}^2} F^{\sharp} = 0$ holds on $D_0(1)$. Hence, the following relation holds on $D_0(1)$.
\begin{equation}\label{boring7}
\triangle_{\mathbb{R}^2} \big ( (F^{\sharp})^2 \big ) = 2 \big | \nabla_{\mathbb{R}^2} F^{\sharp} \big |^2 .
\end{equation}
On the other hand, we have
\begin{equation}\label{boring8}
F^{\sharp} \nabla_{\mathbb{R}^2} \eta^{\sharp} \cdot \nabla_{\mathbb{R}^2} F^{\sharp}
= \frac{1}{2} \Div \Big \{ \eta^{\sharp} \nabla_{\mathbb{R}^2} \big ( (F^{\sharp})^2 \big )  \Big \} -\frac{1}{2} \eta^{\sharp} \triangle_{\mathbb{R}^2} \big ( (F^{\sharp})^2 \big ).
\end{equation}
\eqref{boring7} and \eqref{boring8} together give
\begin{equation}\label{boring9}
F^{\sharp} \nabla_{\mathbb{R}^2} \eta^{\sharp} \cdot \nabla_{\mathbb{R}^2} F^{\sharp}
= \frac{1}{2} \Div \Big \{ \eta^{\sharp} \nabla_{\mathbb{R}^2} \big ( (F^{\sharp})^2 \big )  \Big \}
- \eta^{\sharp} \big | \nabla_{\mathbb{R}^2} F^{\sharp} \big |^2 .
\end{equation}
Since
\begin{equation*}
\begin{split}
& \eta^{\sharp} \Big |_{ \big  \{ |y| = \tanh (\frac{a}{2}R_0) \big  \}} = 1 , \\
& \supp \eta^{\sharp} \subset \Big \{ y \in \mathbb{R}^2 : |y| \leq \tanh \big (  {2}a R_0  \big )   \Big \} , \end{split}
\end{equation*}
it follows that we have
\begin{equation}\label{boring10}
\begin{split}
& \int_{\big \{  \tanh (\frac{a}{2} R_0) < |y| < 1  \big \}} \Div \Big \{ \eta^{\sharp} \nabla_{\mathbb{R}^2} \big ( (F^{\sharp})^2 \big )  \Big \} \dd y^1 \dd y^2\\
=&  - \int_{\big \{ |y| = \tanh (\frac{a}{2} R_0) \big \}} \nabla_{\mathbb{R}^2} \big ( (F^{\sharp})^2  \big ) \cdot \frac{y}{|y|} \dd S .
\end{split}
\end{equation}
However, \eqref{boring7} implies that
\begin{equation}\label{boring11}
\begin{split}
& \int_{\big \{ |y| = \tanh (\frac{a}{2} R_0) \big \}} \nabla_{\mathbb{R}^2} \big ( (F^{\sharp})^2  \big )\cdot \frac{y}{|y|} \dd S  \\
= & \int_{D_0(\tanh(\frac{a}{2}R_0))} \Div_{\mathbb{R}^2} \nabla_{\mathbb{R}^2} \big ( (F^{\sharp})^2  \big ) \dd y^1  \dd y^2 \\
=& \int_{D_0(\tanh(\frac{a}{2}R_0))} \triangle_{\mathbb{R}^2} \big ( (F^{\sharp})^2  \big ) \dd y^1  \dd y^2 \\
=& \int_{D_0(\tanh(\frac{a}{2}R_0))} 2 \big | \nabla_{\mathbb{R}^2} F^{\sharp} \big |^2 \dd y^1  \dd y^2 .
\end{split}
\end{equation}
So \eqref{boring10} and \eqref{boring11} give
\begin{equation}\label{boring12}
\begin{split}
& \frac{1}{2}\int_{\big \{  \tanh (\frac{a}{2} R_0) < |y| < 1  \big \}} \Div \Big \{ \eta^{\sharp} \nabla_{\mathbb{R}^2} \big ( (F^{\sharp})^2 \big )  \Big \}  \dd y^1 \dd y^2 \\
= & - \int_{D_0(\tanh(\frac{a}{2}R_0))}  \big | \nabla_{\mathbb{R}^2} F^{\sharp} \big |^2 \dd y^1  \dd y^2 .
\end{split}
\end{equation}
So, through taking \eqref{boring12} into our consideration, we can now integrate \eqref{boring9} to obtain
\begin{equation}\label{boring13}
\begin{split}
& \int_{\big \{ \tanh (\frac{a}{2} R_0) < |y| < 1 \big \}} F^{\sharp} \nabla_{\mathbb{R}^2} \eta^{\sharp} \cdot \nabla_{\mathbb{R}^2} F^{\sharp} \dd y^1  \dd y^2 \\
= & - \int_{D_0(\tanh(\frac{a}{2}R_0))}  \big | \nabla_{\mathbb{R}^2} F^{\sharp} \big |^2 \dd y^1  \dd y^2
- \int_{\big \{ \tanh (\frac{a}{2} R_0) < |y| < 1 \big \}} \eta^{\sharp} \big | \nabla_{\mathbb{R}^2} F^{\sharp} \big |^2 \dd y^1  \dd y^2 \\
= & - \int_{D_0(1)} \eta^{\sharp} \big | \nabla_{\mathbb{R}^2} F^{\sharp} \big |^2 \dd y^1  \dd y^2 ,
\end{split}
\end{equation}
where the last equality follows simply because $\eta^{\sharp}\big |_{D_0(\tanh (\frac{a}{2}R_0) )} = 1$.
However, 
\begin{equation}\label{boring14}
\int_{D_0(1)} \eta^{\sharp} \big | \nabla_{\mathbb{R}^2} F^{\sharp} \big |^2 \dd y^1  \dd y^2
= \int_{\mathbb{H}^2(-a^2)} \eta_{R_0} g (\dd F , \dd F ) \Vol_{\mathbb{H}^2(-a^2)}.
\end{equation}
So, \eqref{boring13} and \eqref{boring14} together give
\begin{equation}\label{boringfinal}
\begin{split}
&\int_{\big \{ \tanh (\frac{a}{2} R_0) < |y| < 1 \big \}} F^{\sharp} \nabla_{\mathbb{R}^2} \eta^{\sharp} \cdot \nabla_{\mathbb{R}^2} F^{\sharp} \dd y^1 \dd y^2 \\
= &- \int_{\mathbb{H}^2(-a^2)} \eta_{R_0} g (\dd F , \dd F ) \Vol_{\mathbb{H}^2(-a^2)}
\end{split}
\end{equation}
Finally, through combining \eqref{boring6} with \eqref{boringfinal}, we deduce  \eqref{criteriaineq} as needed. 
\end{proof}

Many thanks to Lemma \ref{judgement}, we can go back to the $1$-form $u$ as specified in \eqref{definitionofU} and show that $u$ is indeed a nontrivial element in $H^1_0(\Omega (R_0))$. To achieve this, we prove the following slightly more general result, which we use again in Section \ref{sectionTWO} to show that the solution to the steady Navier Stokes is also nontrivial.

\begin{lemm}\label{JudgementTWO}
Let $\eta_{R_0}, dF, w$ be as in Lemma \ref{judgement}. 
Then, no matter which element $\widetilde{w} \in \textbf{V}(\Omega (R_0))$ we choose, the following term is always a non-zero element in $H^1_0(\Omega (R_0))$.
\begin{equation}\label{termofourFocus}
\big ( \eta_{R_0} -1  \big ) \dd F + w + \widetilde{w}.
\end{equation}
\end{lemm}
\begin{proof}
Suppose
\begin{equation}\label{fakerelation1}
(\eta_{R_0}-1) \dd F + \widetilde{w} + w  = 0 .
\end{equation}
Since $\widetilde{w} \in \textbf{V} (\Omega (R_0))$,
\begin{equation}\label{trivialTrivial}
\int_{\Omega (R_0)} g (\widetilde{w}  , \dd F ) \Vol_{\mathbb{H}^2(-a^2)} =0.
\end{equation}
Then we can test \eqref{fakerelation1} against $\dd F$ over $\Omega (R_0)$ to obtain
\begin{equation*}
\int_{\Omega (R_0)} g ( w, \dd F   ) \Vol_{\mathbb{H}^2(-a^2)} = \int_{\Omega (R_0)} \big ( 1 - \eta_{R_0}   \big ) g (\dd F , \dd F ) \Vol_{\mathbb{H}^2(-a^2)} > 0 ,
\end{equation*}
which violates \eqref{criteriaineq} as stated in Lemma \ref{judgement}. 
\end{proof}

We can summarize our discussion in the following theorem, which is a detailed version of Theorem \ref{main}.
\begin{theo}
Consider a radially symmetric cut-off function $\eta_{R_0} \in C_c^{\infty} (\mathbb{H}^2(-a^2))$ which satisfies constraint \eqref{cutoffconditiontwo}, and take a nontrivial harmonic function $F$ on $\mathbb{H}^2(-a^2)$ for which we have $\dd F \in \mathbb{F}$. Consider some element $w \in H^1_0( B_O(4R_0) - \overline{B_O(2R_0)}) $ which satisfies \eqref{helpfulhypdivergence} on $\mathbb{H}^2(-a^2)$, and whose existence is ensured already by Lemma \ref{KeyLEMMA}.
Let $\widetilde{w}$ to be the unique element in $\textbf{V} \big ( \Omega (R_0) \big )$ for which \eqref{weakformulationONE} holds for any test $1$-form $\varphi \in \textbf{V} \big ( \Omega (R_0) \big )$. We then take the $1$-form $u$ to be
\begin{equation}\label{DefinitionUsecond}
u = \big (  \eta_{R_0} - 1  \big ) \dd F + w + \widetilde{w} .
\end{equation}
Then, it follows that $u \in H^1_0 (\Omega (R_0))$ and that $u$ is \textbf{nontrivial} on $\Omega (R_0)$. Moreover, $u$ is a solution to the following stationary Stokes equation on $\Omega (R_0)$.
\begin{equation}\label{StokesEquation}
\begin{split}
-\triangle  u + 2a^2 u + \dd p & = 0 ,\\
\dd^* u & = 0,\\
u \big |_{\partial \Omega (R_0)} & = 0 ,
\end{split}
\end{equation}
where the associated pressure $p$ is exactly the same one which appears in system \eqref{StandardNSequation}.
\end{theo}

\subsection{Step 6: Showing the solution is not a potential flow}\label{step6}
Here we prove Theorem \ref{main3}.

We actually prove something stronger.  Suppose $v\in H^1_0(\Omega(R_0))$, and $\dd^\ast v=0$.  If in addition $v$ is a potential flow, then $v=\dd f$ for some function $f$ on $\Omega(R_0)$ that must be harmonic since $v$ is divergence free. Writing $v$ in coordinates we have
\[
v=\partial_1 f \dd Y^1+\partial_2 f \dd Y^2.
\]
Let $f^\sharp=f\circ Y^{-1}$.  Then $f^\sharp$ is harmonic on the annulus $$A=\{ y\in \R^2: \tanh (\frac a2 R_0)<\abs{y}<1\}.$$  Next, we can let 
\[
w=w_1 \dd y^1+w_2 \dd y^2,
\]
where
\[
w_1=\partial_1 f,\quad w_2=-\partial_2 f.
\]

Since $f^\sharp$ is harmonic, it follows 
\[
\partial_1w_1=\partial^2_1 f^\sharp=-\partial_2^2 f^\sharp=\partial_2 w_2.
\]

And just by commuting the derivatives that $\partial_2 w_1=-\partial_1 w_2.$  Then
\[
F=w_1+iw_2,
\]
must be analytic on $A$.  Observe, we also have $w_j,$  $j=1, 2$, is harmonic, and $w_j\equiv 0$ on $\partial D_{\tanh(\frac a2 R_0)}(0)$.   By elliptic theory (see for example \cite[Theorem 8.30]{GilbargTrudinger}) $w_j \in C(\{\tanh(\frac a2 R_0)\leq \abs{z}\leq r_0 \})$ for any $r_0$ satisfying $0<\tanh (\frac a2 R_0)<r_0<1$.  So $F$ is continuous as well on $\{\tanh(\frac a2 R_0)\leq \abs{z}\leq r_0 \}$, and
\be\label{F0}
F\equiv 0\quad\mbox{on}\quad \{\abs{z}=\tanh(\frac a2 R_0)\}.
\ee

On the other hand, because $F$ is analytic in $A$, by the Laurent series expansion, for any $z\in \{\tanh(\frac a2 R_0)\leq \abs{z}\leq r_0 \}$,
\[
F(z)=\sum_{k=-\infty}^{+\infty}a_k z^k ,
\]
 where 
\[
a_k=\frac{1}{2\pi i}\int_{\abs{z}=\delta}\frac{F(z)}{z^{k+1}} \dd z,
\]
for $0<\tanh (\frac a2 R_0)<\delta\leq r_0<1$.

We now estimate $a_k$.
\[
\abs{a_k}=\abs{\frac{1}{2\pi \delta^k }\int_0^{2\pi} \frac{F(\delta e^{i\theta})}   {e^{ik\theta}}  \dd \theta}\leq \frac{1}{2\pi \delta^k}\int^{2\pi}_0 
\abs{F(\delta e^{i\theta})} \dd \theta \rightarrow 0,
\]
as $\delta \rightarrow \tanh(\frac a2 R_0)$ by continuity of $F$ and \eqref{F0}.  Since this is true for all $a_k$, $F$ must be trivial everywhere on $\{\tanh(\frac a2 R_0)\leq \abs{z}\leq r_0 \}$.  Because $r_0$ was arbitrary, $F$ is trivial everywhere on $A$, and hence $v$ is trivial on $\Omega(R_0)$, but we know that $v$ is nontrivial.  So $v$ cannot be a potential flow.

\section{Existence of $H^1_0$-Stationary Navier-Stokes flow in the exterior domain}\label{sectionTWO}

In this section, we demonstrate how to obtain a nontrivial $H^1_0$ solution to the following system on the exterior domain $\Omega (R_0)$.
\begin{equation}\label{TrueNavierStokeseq}
\begin{split}
 -\triangle  v + 2a^2 v + \nabla_v v + \dd P & = 0 , \\
\dd^* v & = 0 , \\
v \big |_{\partial \Omega (R_0)} & = 0 .
\end{split}
\end{equation}

We  split the argument into different steps as follows.
\subsection{About our preferred form of the solution to \eqref{TrueNavierStokeseq}.  }\label{subsection4.1}
Motivated by our success in Section \ref{sectionONE}, the nontrivial solution $v$ to \eqref{TrueNavierStokeseq} on $\Omega (R_0)$ will take the form

\begin{equation}\label{FormoftheSol}
v = \big ( \eta_{R_0} - 1 \big ) \dd F + w + \widetilde{w} ,
\end{equation}
where $F$ is a nontrivial harmonic function on $\mathbb{H}^2(-a^2)$ for which $\dd F \in \mathbb{F}$,
$w$ is the element as specified in Lemma \ref{KeyLEMMA}, and $\widetilde{w} \in \textbf{V}\big ( \Omega (R_0) \big )$.
Now, saying that the element $v$ as given in \eqref{FormoftheSol} is a solution to \eqref{TrueNavierStokeseq} is the same as saying that $\widetilde{w} \in \textbf{V} \big ( \Omega (R_0)  \big )$ is a solution to the following system on
$\Omega (R_0)$.
\begin{equation}\label{TureNSSecond}
\begin{split}
-\triangle \widetilde{w}  + 2a^2 \widetilde{w} + \nabla_{\widetilde{w}} \widetilde{w} + \nabla_{\widetilde{w}} \Psi
+ \nabla_{\Psi} \widetilde{w} + \dd P  & = \Phi , \\
\dd^* \widetilde{w} & = 0 ,\\
\widetilde{w} \big |_{\partial \Omega (R_0)} & = 0,
\end{split}
\end{equation}
where
\begin{equation}\label{PsiandPhi}
\begin{split}
\Psi & = \big ( \eta_{R_0} - 1  \big ) \dd F  + w , \\
\Phi & = \triangle \Psi - 2a^2 \Psi - \nabla_{\Psi} \Psi .
\end{split}
\end{equation}

\subsection{The estimates for $\|w\|_4$, $\|\nabla w\|_2$,  $\|\Psi\|_4$, $\|\nabla \Psi\|_2$. }
To find $\tilde w$, we first need some estimates.  We estimate $\|w\|_{L^4(\mathbb{H}^2(-a^2))}$, $\|\nabla w \|_{L^2(\mathbb{H}^2(-a^2))}$, as well as $\|\Psi\|_{L^4(\mathbb{H}^2(-a^2))}$ , $\| \nabla \Psi \|_{L^2(\mathbb{H}^2(-a^2))}$, all in terms of $\|\dd F\|_{L^2(\mathbb{H}^2(-a^2))}$. We use the following results.
\begin{lemm}\cite{CC13}
For any harmonic function $F$ on $\mathbb{H}^2(-a^2)$ for which $\dd F \in \mathbb{F}$, the following a priori estimate holds.
\begin{equation}\label{estimateharmonic}
 \big \| \nabla \dd F \big \|_{L^2(\mathbb{H}^2(-a^2))}  \leq C_a \big \| \dd F \big \|_{L^2(\mathbb{H}^2(-a^2))} .
\end{equation}
\end{lemm}

\begin{lemm}\cite{CC13}
The following a priori estimate holds for any $1$-form $\varphi \in H^1_0 (\mathbb{H}^2(-a^2))$.
\begin{equation}\label{Ladyzhenskaya}
\big \| \varphi \big \|_{L^4 (\mathbb{H}^2(-a^2))} \leq C_a \big \| \varphi \big \|_{L^2 (\mathbb{H}^2(-a^2))}^{\frac{1}{2}} \Big \{ \big \| \varphi \big \|_{L^2 (\mathbb{H}^2(-a^2))} + \big \| \nabla \varphi \big \|_{L^2 (\mathbb{H}^2(-a^2))}  \Big  \}^{\frac{1}{2}} .
\end{equation}
\end{lemm}

\begin{lemm}\cite{CC15}
 Let $\varphi$ be a $1-$form  in $H^1_0(\mathbb{H}^2(-a^2))$, then
\begin{equation}\label{PoincareType}
\big \| \varphi \big \|_{L^2(\mathbb{H}^2(-a^2))} \leq \frac{1}{a} \big \| \nabla \varphi \big \|_{L^2(\mathbb{H}^2(-a^2))} .
\end{equation}
\end{lemm}
So, through combining \eqref{Ladyzhenskaya} with \eqref{PoincareType}, we obtain  
\begin{equation}\label{reallyEASY}
\big \| \varphi \big \|_{L^4(\mathbb{H}^2(-a^2))} \leq C_a \big \| \nabla \varphi \big \|_{L^2(\mathbb{H}^2(-a^2))} .
\end{equation}
Let
\begin{equation}\label{finiteregion}
\Omega (r_1 , r_2) = \big \{ x \in \mathbb{H}^2(-a^2) : r_1 < \rho (x) < r_2  \big \} ,
\end{equation}
where $0 < r_1 < r_2$.

In what follows, the absolute constants $C_a$, and $C(a, R_0)$ may change from line to line.   

Recall that $w$ is the element which is specified in Lemma \ref{KeyLEMMA}.  By \eqref{reallyEASY} and \eqref{H1estimateforW} 
\begin{equation}\label{HelpfulExtra1}
\big \| w \big \|_{L^4(\mathbb{H}^2(-a^2))} \leq C(a, R_0) \big \| \dd F \big \|_{L^2(\mathbb{H}^2(-a^2))} .
\end{equation}

Also, \eqref{reallyEASY}, \eqref{estimateharmonic} imply

\begin{equation}\label{helpfuldF}
\begin{split}
\big \| \dd F \big \|_{L^4(\mathbb{H}^2(-a^2))} & \leq  C_a \big \| \nabla \dd F \big \|_{L^2(\mathbb{H}^2(-a^2))} \\
& \leq C_a \big \| \dd F \big \|_{L^2(\mathbb{H}^2(-a^2))}.
\end{split}
\end{equation}
For the term $\Psi$ as specified in \eqref{PsiandPhi}, we have

\begin{equation}\label{weired2}
\begin{split}
\big \| \Psi \big \|_{L^4(\mathbb{H}^2(-a^2))} & \leq  \big \| \dd F \big \|_{L^4(\mathbb{H}^2(-a^2))} +
\big \| w \big \|_{L^4(\mathbb{H}^2(-a^2))} \\
& \leq C_a \big \|  \dd F \big \|_{L^2(\mathbb{H}^2(-a^2))} + \big \| w \big \|_{L^4(\mathbb{H}^2(-a^2))} \\
& \leq  C(a,R_0) \big \| \dd F \big \|_{L^2(\mathbb{H}^2(-a^2))} ,
\end{split}
\end{equation}
where \eqref{helpfuldF} is used in the second line and \eqref{HelpfulExtra1} is used in the third line.
Observe 
\begin{equation*}
\big \| \dd  \eta_{R_0} \big \|_{L^{\infty} (\mathbb{H}^2(-a^2))} \leq \frac{1}{2R_0} \big \| \eta'\big \|_{L^{\infty} ([0,\infty ))} ,
\end{equation*}
and 
\begin{equation}
\nabla \Psi = \dd \eta_{R_0} \otimes \dd F + (\eta_{R_0}-1) \nabla \dd F + \nabla w .
\end{equation}
So, through applying \eqref{estimateharmonic} and \eqref{H1estimateforW} we have
\begin{equation}\label{Weired1}
\begin{split}
& \big \| \nabla \Psi \big \|_{L^2(\mathbb{H}^2(-a^2))}\\
\leq &\frac{1}{2R_0} \big \| \eta'\big \|_{L^{\infty} ([0,\infty ))} \big \| \dd F \big \|_{L^2(\mathbb{H}^2(-a^2))} + \big \| \nabla \dd F \big \|_{L^2(\mathbb{H}^2(-a^2))}
+ \big \| \nabla w \big \|_{L^2(\mathbb{H}^2(-a^2))} \\
\leq & C(a ,R_0 ) \big \| \dd F \big \|_{L^2(\mathbb{H}^2(-a^2))}.
\end{split}
\end{equation}
Then by \eqref{H1estimateforW}, \eqref{HelpfulExtra1}, \eqref{weired2}, \eqref{Weired1} 
\begin{equation}\label{EntireONE}
\begin{split}
& \big \| w \big \|_{L^4(\mathbb{H}^2(-a^2))} + \big \| \nabla  w \big \|_{L^2(\mathbb{H}^2(-a^2))}
+ \big \|\Psi \big \|_{L^4(\mathbb{H}^2(-a^2))} + \big \| \nabla \Psi \big \|_{L^2(\mathbb{H}^2(-a^2))} \\
&\quad\leq  C(a, R_0 ) \big \| \dd F \big \|_{L^2(\mathbb{H}^2(-a^2))} .
\end{split}
\end{equation}

Of course, \eqref{PoincareType} and \eqref{EntireONE} together give
\begin{equation}\label{entireTWO}
\big \| w  \big \|_{L^2(\mathbb{H}^2(-a^2))}  +   \big \| \Psi \big \|_{L^2(\mathbb{H}^2(-a^2))}
\leq  C(a, R_0 ) \big \| \dd F \big \|_{L^2(\mathbb{H}^2(-a^2))} .
\end{equation}

\subsection{The set up of approximated solutions on larger and larger bounded subregions $\Omega (R_0 , R)$ of $\Omega (R_0)$. }

Now, let $R$  be any positive number which satisfies $R > 5R_0$. We introduce the following function space
\begin{equation}\label{generalV}
\textbf{V} \big ( \Omega (R_0 , R )  \big ) = \overline{\Lambda_{c, \sigma }\big ( \Omega (R_0 , R ) \big )}^{ \|\cdot  \|_{H^1}} ,
\end{equation}
where
\[
\Omega (R_0 , R)  = \big \{ x \in \mathbb{H}^2(-a^2) : R_0 < \rho (x) < R     \big \} .
\]
Note that each element in $\tV \big ( \Omega (R_0 , R )  \big )$ can be thought of as an element in $\tV\big (\Omega (R_0) \big )$. 
 
Following \cite{Seregin},  we look for an element $w_R \in \textbf{V} \big ( \Omega (R_0 , R )  \big )$ which solves \eqref{TureNSSecond} on $\Omega (R_0 , R )$.  More precisely we look for $w_R$ such that
\begin{equation}\label{approxNSONE}
\begin{split}
  -\triangle   w_R + 2a^2 w_R + \nabla_{w_R} w_R + \nabla_{w_R} \Psi + \nabla_{\Psi} w_R + \dd P_R & = \Phi ,\\
 \dd^* w_R & = 0 ,\\
 w_R \big |_{\partial \Omega (R_0 , R )} & = 0,
\end{split}
\end{equation}
and we say that   $w_R \in \tV \big ( \Omega (R_0 , R )  \big )$ is a solution  to  \eqref{approxNSONE}  if the following holds for any test $1$-form
$\varphi \in \textbf{V} \big ( \Omega (R_0 , R )  \big )$.

\begin{equation}\label{approxWEAK}
\begin{split}
& \int_{\Omega (R_0 , R)} g (\dd w_R , \dd \varphi ) \Vol_{\mathbb{H}^2(-a^2)} + 2a^2 \int_{\Omega (R_0 , R)}
g(w_R, \varphi ) \Vol_{\mathbb{H}^2(-a^2)} \\
&= \Big < \Phi , \varphi \Big >_{\textbf{V}'  ( \Omega (R_0 , R )  )\otimes \textbf{V} ( \Omega (R_0 , R )  )}
- \int_{\Omega (R_0 , R)} g ( \nabla_{w_R} \Psi  , \varphi  )\Vol_{\mathbb{H}^2(-a^2)} \\
& \quad- \int_{\Omega (R_0 , R )} g (\nabla_{\Psi} w_R , \varphi ) \Vol_{\mathbb{H}^2(-a^2)}
-\int_{\Omega (R_0 , R )}  g (\nabla_{w_R} w_R , \varphi ) \Vol_{\mathbb{H}^2(-a^2)}.
\end{split}
\end{equation}
Before we can address the existence of ${w_R}$, which satisfies the weak formulation \eqref{approxWEAK}, we need to establish an a priori estimate for $\|\nabla {w_R} \|_{L^2(\Omega (R_0))}$, under some suitable constraint imposed on $\dd F \in \mathbb{F}$.

\subsection{Uniform a priori estimate for the $H^1$ norm of $w_R$.}\label{subsection3.4}
Now, we need to see how to get a uniform estimate of the following quantity in a manner which is independent of the parameter $R$.
\begin{equation*}
\big \| \nabla w_R \big \|_{L^2(\Omega (R_0 , R ))}^2 = \big \| \dd w_R \big \|_{L^2(\Omega (R_0 , R))}^2 + a^2
\big \| w_R \big \|_{L^2( \Omega (R_0 , R ) )}^2 .
\end{equation*}
The first step  is to set $\varphi$ to be just $w_R$ itself in  \eqref{approxWEAK}. This leads to
\begin{equation}\label{substituteWR}
\begin{split}
& \big \| \nabla w_R \big \|_{L^2(\Omega ( R_0 , R ))}^2 + a^2 \big \| w_R \big \|_{L^2(\Omega (R_0 , R))}^2 \\
&=  \big \| \dd w_R \big \|_{L^2(\Omega ( R_0 , R ))}^2 + 2a^2 \big \| w_R \big \|_{L^2(\Omega (R_0 , R))}^2 \\
&=  \Big < \Phi , w_R \Big >_{\textbf{V}'  ( \Omega (R_0 , R )  )\otimes \textbf{V} ( \Omega (R_0 , R )  )}
- \int_{\Omega (R_0 , R)} g ( \nabla_{w_R} \Psi  , w_R  )\Vol_{\mathbb{H}^2(-a^2)} \\
 &\quad - \int_{\Omega (R_0 , R )} g (\nabla_{\Psi} w_R , w_R ) \Vol_{\mathbb{H}^2(-a^2)}
-\int_{\Omega (R_0 , R )}  g (\nabla_{w_R} w_R , w_R ) \Vol_{\mathbb{H}^2(-a^2)}.
\end{split}
\end{equation}

Now, because
\begin{equation}\label{NatInclusion}
\textbf{V} \big ( \Omega (R_0 , R )  \big ) \subset \textbf{V} \big (  \Omega  (R_0)    \big ) \subset \textbf{V} \big ( \mathbb{H}^2(-a^2) \big ) ,
\end{equation}
we have \cite[Lemma 4.2]{CC13}
\begin{equation}\label{nonlinear}
\int_{\Omega (R_0 , R )}  g (\nabla_{w_R} w_R , w_R ) \Vol_{\mathbb{H}^2(-a^2)} = 0.
\end{equation}
Also, observe that $\dd^* \Psi = 0$ holds on $\mathbb{H}^2(-a^2)$. So, we have
\begin{equation}\label{suprise}
\Psi \in \big \{  \varphi \in H^1_0(\mathbb{H}^2(-a^2)) : \dd^* \varphi = 0 \big \}.
\end{equation}
So again by \cite[Lemma 4.2]{CC13} we get
\begin{equation}\label{surprise2}
\int_{\Omega (R_0 , R )} g (\nabla_{\Psi} w_R , w_R ) \Vol_{\mathbb{H}^2(-a^2)} = 0 .
\end{equation}
Then   \eqref{nonlinear} and \eqref{surprise2} reduce \eqref{substituteWR} to
\begin{equation}\label{genuineidentity}
\begin{split}
& \big \| \nabla w_R \big \|_{L^2(\Omega ( R_0 , R ))}^2 + a^2 \big \| w_R \big \|_{L^2(\Omega (R_0 , R))}^2 \\
 &= \Big < \Phi , w_R \Big >_{\textbf{V}'  ( \Omega (R_0 , R )  )\otimes \textbf{V} ( \Omega (R_0 , R )  )}
- \int_{\Omega (R_0 , R)} g ( \nabla_{w_R} \Psi  , w_R  )\Vol_{\mathbb{H}^2(-a^2)} .
\end{split}
\end{equation}
 
Next, we estimate
\begin{equation}\label{TEDIOUS1}
\begin{split}
& \Big < \Phi , \varphi \Big >_{\textbf{V}'(\mathbb{H}^2(-a^2)) \otimes \textbf{V} (\mathbb{H}^2(-a^2)) } \\
& = \Big < \triangle  \Psi , \varphi \Big >_{\textbf{V}'(\mathbb{H}^2(-a^2)) \otimes \textbf{V} (\mathbb{H}^2(-a^2)) }
-2a^2 \int_{\mathbb{H}^2(-a^2)} g (\Psi , \varphi ) \Vol_{\mathbb{H}^2(-a^2)} \\
&\quad - \int_{\mathbb{H}^2(-a^2)} g (\nabla_{\Psi} \Psi , \varphi ) \Vol_{\mathbb{H}^2(-a^2)}\\
 &= - \int_{\mathbb{H}^2(-a^2)} g ( \dd \Psi , \dd \varphi ) \Vol_{\mathbb{H}^2(-a^2)}
-2a^2 \int_{\mathbb{H}^2(-a^2)} g (\Psi , \varphi ) \Vol_{\mathbb{H}^2(-a^2)}\\
 &\quad- \int_{\mathbb{H}^2(-a^2)} g (\nabla_{\Psi} \Psi , \varphi ) \Vol_{\mathbb{H}^2(-a^2)} .
\end{split}
\end{equation}

Through applying \eqref{EntireONE}, we get
\begin{equation}\label{EXTRA1}
\begin{split}
\Big | \int_{\mathbb{H}^2(-a^2)} g (\dd \Psi , \dd \varphi ) \Vol_{\mathbb{H}^2(-a^2)} \Big | & \leq C(a, R_0)
\big \| \dd F \big \|_{L^2(\mathbb{H}^2(-a^2))}  \big \| \nabla \varphi \big \|_{L^2(\mathbb{H}^2(-a^2))}. 
\end{split}
\end{equation}

Of course, \eqref{entireTWO} gives

\begin{equation}\label{Extra2}
\Big | \int_{\mathbb{H}^2(-a^2)} g (\Psi , \varphi ) \Vol_{\mathbb{H}^2(-a^2)} \Big | \leq C(a, R_0)
\big \| \dd F \big \|_{L^2(\mathbb{H}^2(-a^2))}  \big \| \nabla \varphi \big \|_{L^2(\mathbb{H}^2(-a^2))}.
\end{equation}

Again, through using \eqref{EntireONE} and then \eqref{reallyEASY} successively, we get

\begin{equation}\label{Extra3}
\begin{split}
\Big | \int_{\mathbb{H}^2(-a^2)} g (\nabla_{\Psi} \Psi , \varphi ) \Vol_{\mathbb{H}^2(-a^2)} \Big |
& \leq \big \| \Psi \big \|_{L^4(\mathbb{H}^2(-a^2))}\big \| \nabla  \Psi \big \|_{L^2(\mathbb{H}^2(-a^2))}
\big \| \varphi \big \|_{L^4(\mathbb{H}^2(-a^2))} \\
& \leq C(a, R_0) \big \| \dd F \big \|_{L^2(\mathbb{H}^2(-a^2))}^2       \big \| \nabla \varphi \big \|_{L^2(\mathbb{H}^2(-a^2))} .
\end{split}
\end{equation}

So, through combining \eqref{TEDIOUS1} with \eqref{EXTRA1}, \eqref{Extra2}, and \eqref{Extra3}, we get
\begin{equation}\label{Phidualest1}
\big \| \Phi \big \|_{\textbf{V}'(\mathbb{H}^2(-a^2))} \leq C(a, R_0) \Big \{ \big \| \dd F \big \|_{L^2(\mathbb{H}^2(-a^2))} + \big \| \dd F \big \|_{L^2(\mathbb{H}^2(-a^2))}^2     \Big \}
\end{equation}
The inclusion  \eqref{NatInclusion} tells us that
\begin{equation}\label{InclusionRev}
\textbf{V}' \big ( \mathbb{H}^2(-a^2) \big ) \subset \textbf{V}' \big ( \Omega (R_0) \big ) \subset
\textbf{V}' \big ( \Omega (R_0 , R ) \big ).
\end{equation}
Hence
\begin{equation*}
  \big \| \Phi \big \|_{\textbf{V}' (\Omega (R_0 , R ))}  \leq \big \| \Phi \big \|_{\textbf{V}'(\Omega (R_0))} \leq    \big \| \Phi \big \|_{\textbf{V}' (\mathbb{H}^2(-a^2))},
\end{equation*}
and
\begin{equation}\label{PhiDualEST2}
\big \| \Phi  \big \|_{\textbf{V}' ( \Omega (R_0 , R)  )} \leq C(a, R_0) \Big \{ \big \| \dd F \big \|_{L^2(\mathbb{H}^2(-a^2))} + \big \| \dd F \big \|_{L^2(\mathbb{H}^2(-a^2))}^2     \Big \}.
\end{equation}
By applying \eqref{EntireONE} to $\nabla \Psi$ and then \eqref{reallyEASY} to $w_R$, we can estimate the second term on the right hand side of \eqref{genuineidentity} as follows.
\begin{equation}\label{Ninety}
\begin{split}
 \Big |  \int_{\Omega (R_0 , R ) } g (  \nabla_{w_R} \Psi , w_R      )  \Vol_{\mathbb{H}^2(-a^2)}       \Big | 
\leq  C(a, R_0) \big \| \nabla w_R \big \|_{L^2( \Omega ( R_0 , R   )   )}^2 \big \| \dd F \big \|_{L^2(\mathbb{H}^2(-a^2))} .
\end{split}
\end{equation}
Then it follows
\begin{equation*}
\begin{split}
& \big \| \nabla w_R \big \|_{L^2(\Omega (R_0 , R ))}^2 + a^2 \big \| w_R \big \|_{L^2(\Omega (R_0 , R))}^2 \\
&\leq  \big \| \Phi \big \|_{\textbf{V}'(\Omega (R_0 , R))} \big \| \nabla w_R \big \|_{L^2(\Omega (R_0 , R))}
+  {C}(a, R_0) \big \| \nabla w_R \big \|_{L^2(\Omega (R_0 ,R))}^2 \big \| \dd F \big \|_{L^2(\mathbb{H}^2(-a^2))} \\
&\leq  \frac{1}{2}   \big \| \nabla w_R \big \|_{L^2(\Omega (R_0 , R ))}^2 + \frac{1}{2} \big \| \Phi \big \|_{\textbf{V}'(\Omega (R_0 , R))}^2
+  {C}(a, R_0) \big \| \nabla w_R \big \|_{L^2(\Omega (R_0 ,R))}^2 \big \| \dd F \big \|_{L^2(\mathbb{H}^2(-a^2))} \\
&\leq  \frac{1}{2}   \big \| \nabla w_R \big \|_{L^2(\Omega (R_0 , R ))}^2
+ \frac{1}{2}  \big ( {C}(a,R_0) \big )^2  \Big \{ \big \| \dd F \big \|_{L^2(\mathbb{H}^2(-a^2))} + \big \| \dd F \big \|_{L^2(\mathbb{H}^2(-a^2))}^2     \Big \}^2 \\
&\quad+  {C}(a, R_0) \big \| \nabla w_R \big \|_{L^2(\Omega (R_0 ,R))}^2 \big \| \dd F \big \|_{L^2(\mathbb{H}^2(-a^2))} ,
\end{split}
\end{equation*}
which gives
\begin{equation}\label{decision}
\begin{split}
& \Big \{ \frac{1}{2} -  {C}(a,R_0) \big \|  \dd F \big \|_{L^2(\mathbb{H}^2(-a^2))}   \Big \}  \big \| \nabla w_R \big \|_{L^2(\Omega (R_0 ,R))}^2 + a^2  \big \|  w_R \big \|_{L^2(\Omega (R_0 ,R))}^2 \\
\leq & \frac{1}{2} \big ( {C}(a,R_0) \big )^2  \Big \{ \big \| \dd F \big \|_{L^2(\mathbb{H}^2(-a^2))} + \big \| \dd F \big \|_{L^2(\mathbb{H}^2(-a^2))}^2     \Big \}^2.
\end{split}
\end{equation}

The a priori estimate \eqref{decision} tells us that we have to post the following constraint on the size of the $L^2$-norm of the harmonic $1$-form $\dd F$.

\begin{equation}\label{ConditionondF}
\big \| \dd F \big \|_{L^2(\mathbb{H}^2(-a^2))}  < \frac{1}{2 {C}(a ,R_0 )} .
\end{equation}
 
Consequently, suppose that $\dd F$ satisfies  \eqref{ConditionondF}, then the element $w_R$, which is a solution to the system \eqref{approxNSONE} on $\Omega (R_0 , R)$, must satisfy the following a priori estimate
\begin{equation*}
\big \| \nabla w_R \big \|_{L^2(\Omega (R_0 ,R))}^2 \leq
\frac{\big ( {C}(a,R_0) \big )^2 \Big \{ \big \| \dd F \big \|_{L^2(\mathbb{H}^2(-a^2))} + \big \| \dd F \big \|_{L^2(\mathbb{H}^2(-a^2))}^2     \Big \}^2  }{\Big ( 1- 2 {C}(a,R_0)  \|\dd F \|_{L^2(\mathbb{H}^2(-a^2))} \Big ) }
 .
\end{equation*}
 
We summarize our discussion in the following lemma.

\begin{lemm}\label{ImportantLemma}
There exists some positive absolute constant ${C} (a, R_0)$, which depends only on $a$ and $R_0$, such that for any arbitrary chosen positive number $R$ which satisfies $R > 5R_0$, and any arbitrary chosen harmonic function $F$ on $\mathbb{H}^2(-a^2)$ for which $\dd F \in \mathbb{F}$, the following implication holds for any possible element
$w_R \in \textbf{V} (\Omega (R_0 ,R ))$, which is a solution to the system \eqref{approxNSONE} on $\Omega (R_0 , R)$.\\

If  $\dd F$ satisfies 
\begin{equation}\label{ConditionondFTWO}
\big \| \dd F \big \|_{L^2(\mathbb{H}^2(-a^2))}  < \frac{1}{2 {C}(a ,R_0 )} ,
\end{equation}
then, it follows that $w_R$ must satisfy the following a priori estimate.
\begin{equation}\label{GOODESTIMATE}
\big \| \nabla w_R \big \|_{L^2(\Omega (R_0 ,R))}^2 \leq
\frac{\big ( {C}(a,R_0) \big )^2 \Big \{ \big \| \dd F \big \|_{L^2(\mathbb{H}^2(-a^2))} + \big \| \dd F \big \|_{L^2(\mathbb{H}^2(-a^2))}^2     \Big \}^2  }{\Big ( 1- 2 {C}(a,R_0)  \|\dd F \|_{L^2(\mathbb{H}^2(-a^2))} \Big ) } .
\end{equation}
\end{lemm}

\subsection{The existence of  $w_R$ under  \eqref{ConditionondFTWO} imposed on $\dd F$.}

The purpose of this section is just to state the following result.

\begin{lemm}\label{ReallyStandardresult}
Let ${C}(a,R_0) >0 $  be the same absolute constant as specified in Lemma \ref{ImportantLemma}. Then, it follows that, no matter which $R > 5 R_0$ we use, the weak formulation \eqref{approxWEAK} admits at least one solution $w_R \in \textbf{V} (\Omega (R_0 , R))$, provided that $\dd F \in \mathbb{F}$ satisfies the constraint \eqref{ConditionondFTWO}. Moreover, such  $w_R$ satisfies the a priori estimate \eqref{GOODESTIMATE}.
\end{lemm}

The proof of Lemma \ref{ReallyStandardresult} follows from an application of the Leray-Schauder fixed point argument. Since the proof of Lemma \ref{ReallyStandardresult} uses   routine estimates which we will demonstrate in Section \ref{EASYsubsection}, we give the proof of Lemma \ref{ReallyStandardresult} later, in  Section \ref{RoutineSubsection}.

\subsection{Getting a solution to \eqref{TureNSSecond} on $\Omega (R_0)$ through passing to the limit on  $w_R$. }\label{approx}

Here, we take any nontrivial harmonic function $F$ on $\mathbb{H}^2(-a^2)$ such that $\dd F$ satisfies   \eqref{ConditionondFTWO}. With respect to such a harmonic form $\dd F \in \mathbb{F}$, we consider for each $R > 5R_0$, an element $w_R \in V\big (\Omega (R_0 , R) \big )$, which is a solution to \eqref{approxNSONE} on $\Omega (R_0 , R)$. Then, Lemma \ref{ImportantLemma} informs us that each $w_R$ must satisfy the apriori estimate \eqref{GOODESTIMATE}.
 
Now, \eqref{GOODESTIMATE} informs us that the family $\{w_R : R > 5R_0\}$ is bounded in the Hilbert space $\textbf{V}\big ( \Omega (R_0) \big )$. So, we can find a strictly increasing sequence
$\{ R_m \}_{m=1}^{\infty}$ of positive numbers in $(5R_0 , \infty )$ such that
\begin{equation}\label{limitofR}
\lim_{m \rightarrow \infty } R_m = + \infty ,
\end{equation}
and that the sequence $ \big \{ w_{R_m} \big \}_{m=1}^{\infty}$ converges weakly in the Hilbert space $\textbf{V}\big ( \Omega (R_0)\big )$ to some limiting element $\widetilde{w}$ in $\textbf{V} \big ( \Omega (R_0) \big )$. Hence, for any $\varphi \in \textbf{V} \big ( \Omega (R_0) \big )$, we have
\begin{equation}\label{Weakconvergence}
\lim_{m\rightarrow \infty} (( w_{R_m} , \varphi   ))_{H^1_0(\Omega (R_0))} = (( \widetilde{w}  , \varphi  ))_{H^1_0(\Omega (R_0))} .
\end{equation}
Note that the limiting element $\widetilde{w}$ satisfies \eqref{GOODESTIMATE} also.
The main task here is to show that $\widetilde{w}$ is a solution to \eqref{TureNSSecond} on $\Omega (R_0)$.
In other words, we have to prove that the following relation holds for any given test $1$-form
$\varphi \in \Lambda_{c,\sigma}^1 \big ( \Omega (R_0)\big ) $.

\begin{equation}\label{WEAKReally}
\begin{split}
& \int_{\Omega (R_0)} g (\dd \widetilde{w} , \dd \varphi ) \Vol_{\mathbb{H}^2(-a^2)} + 2a^2 \int_{\Omega (R_0)}
g(\widetilde{w}, \varphi ) \Vol_{\mathbb{H}^2(-a^2)} \\
= & \Big < \Phi , \varphi \Big >_{\textbf{V}'  ( \Omega (R_0)  )\otimes \textbf{V} ( \Omega (R_0)  )}
- \int_{\Omega (R_0)} g ( \nabla_{\widetilde{w}} \Psi  , \varphi  )\Vol_{\mathbb{H}^2(-a^2)} \\
& - \int_{\Omega (R_0)} g (\nabla_{\Psi} \widetilde{w} , \varphi ) \Vol_{\mathbb{H}^2(-a^2)}
-\int_{\Omega (R_0)}  g (\nabla_{\widetilde{w}} \widetilde{w} , \varphi ) \Vol_{\mathbb{H}^2(-a^2)}.
\end{split}
\end{equation}

We present the details for completeness, and also because similar arguments can be used later in Section \ref{RoutineSubsection}; we will omit the details then.

So, we now take a  fixed test $1$-form $\varphi \in \Lambda_{c, \sigma}^1 (\Omega (R_0))$. Since $\varphi$ is compactly supported in the open region $\Omega (R_0)$, we can choose a fixed, sufficiently large radius $\widetilde{R}$ which satisfies $\widetilde{R} > 5 R_0$ such that we have
\begin{equation}\label{supportofvarphi}
\varphi \in \Lambda_{c , \sigma}^1 (\Omega (R_0 , \widetilde{R})).
\end{equation}
Now, due to \eqref{limitofR}, we can find some sufficiently large positive integer $N \in \mathbb{Z}^+$ such that: whenever $m \geq N$, we have $R_m > \widetilde{R}$. Note that the following inclusion is valid for any $m \geq N$
\begin{equation}\label{InclusionTRIVIAL}
\Omega (R_0 , \widetilde{R}) \subset \Omega (R_0 , R_m).
\end{equation}
The validity of \eqref{InclusionTRIVIAL} for all $m \geq N$, together with \eqref{supportofvarphi}, tell us that our fixed choice of test $1$-form $\varphi$ must satisfy the following property for all $m \geq N$.
\begin{equation}\label{veryveryGood}
\varphi \in \textbf{V} (\Omega (R_0 , R_m )) .
\end{equation}
The validity of \eqref{veryveryGood} for all $m \geq N$ allows us to deduce directly from \eqref{approxWEAK} that our fixed choice of  the test $1$-form $\varphi$ satisfy the following identity for all $m \geq N$.
\begin{equation}\label{approxWEAKRm}
\begin{split}
& \int_{\Omega (R_0 , R_m)} g (\dd w_{R_m} , \dd \varphi ) \Vol_{\mathbb{H}^2(-a^2)} + 2a^2 \int_{\Omega (R_0 , R_m)}
g(w_{R_m}, \varphi ) \Vol_{\mathbb{H}^2(-a^2)} \\
= & \Big < \Phi , \varphi \Big >_{\textbf{V}'  ( \Omega (R_0 , R_m )  )\otimes \textbf{V} ( \Omega (R_0 , R_m )  )}
- \int_{\Omega (R_0 , R_m)} g ( \nabla_{w_{R_m}} \Psi  , \varphi  )\Vol_{\mathbb{H}^2(-a^2)} \\
& - \int_{\Omega (R_0 , R_m )} g (\nabla_{\Psi} w_{R_m} , \varphi ) \Vol_{\mathbb{H}^2(-a^2)}
-\int_{\Omega (R_0 , R_m )}  g (\nabla_{w_{R_m}} w_{R_m} , \varphi ) \Vol_{\mathbb{H}^2(-a^2)}.
\end{split}
\end{equation}
However, since the following relation holds for all $m \geq N$,
\begin{equation*}
\supp \varphi \subset \Omega (R_0 , \widetilde{R}) \subset \Omega (R_0 , R_m) ,
\end{equation*}
the following identities definitely hold for any $m \geq N$.
\begin{equation*}
\begin{split}
&\int_{\Omega (R_0 , R_m)} \Big \{  g (\dd w_{R_m} , \dd \varphi )  + 2a^2
g(w_{R_m}, \varphi )   \Big \}  \Vol_{\mathbb{H}^2(-a^2)}  = (( w_{R_m} , \varphi  ))_{H^1_0(\Omega (R_0))} ,\\
&\Big < \Phi , \varphi \Big >_{\textbf{V}'  ( \Omega (R_0 , R_m )  )\otimes \textbf{V} ( \Omega (R_0 , R_m )  )}  =
\Big < \Phi , \varphi \Big >_{\textbf{V}'  ( \Omega (R_0 )  )\otimes \textbf{V} ( \Omega (R_0 )  )} ,\\
& \int_{\Omega (R_0 , R_m)} g ( \nabla_{w_{R_m}} \Psi  , \varphi  )\Vol_{\mathbb{H}^2(-a^2)}
= \int_{\Omega (R_0 , \widetilde{R})} g ( \nabla_{w_{R_m}} \Psi  , \varphi  )\Vol_{\mathbb{H}^2(-a^2)},\\
&\int_{\Omega (R_0 , R_m )} g (\nabla_{\Psi} w_{R_m} , \varphi ) \Vol_{\mathbb{H}^2(-a^2)}
=\int_{\Omega (R_0 , \widetilde{R} )} g (\nabla_{\Psi} w_{R_m} , \varphi ) \Vol_{\mathbb{H}^2(-a^2)} ,\\
&\int_{\Omega (R_0 , R_m )}  g (\nabla_{w_{R_m}} w_{R_m} , \varphi ) \Vol_{\mathbb{H}^2(-a^2)}
= \int_{\Omega (R_0 , \widetilde{R} )}  g (\nabla_{w_{R_m}} w_{R_m} , \varphi ) \Vol_{\mathbb{H}^2(-a^2)}
\end{split}
\end{equation*}
Many thanks to the identities as displayed above, relation \eqref{approxWEAKRm} can now be rephrased in the following one, which of course holds for all $m \geq N$ also.
\begin{equation}\label{RelationForPasstoLimit}
\begin{split}
& (( w_{R_m}  , \varphi    ))_{H^1_0(\Omega (R_0))} \\
= & \Big < \Phi , \varphi \Big >_{\textbf{V}'  ( \Omega (R_0)  )\otimes \textbf{V} ( \Omega (R_0)  )}
- \int_{\Omega (R_0 ,\widetilde{R})} g ( \nabla_{w_{R_m}} \Psi  , \varphi  )\Vol_{\mathbb{H}^2(-a^2)} \\
& - \int_{\Omega (R_0 , \widetilde{R} )} g (\nabla_{\Psi} w_{R_m} , \varphi ) \Vol_{\mathbb{H}^2(-a^2)}
-\int_{\Omega (R_0 , \widetilde{R} )}  g (\nabla_{w_{R_m}} w_{R_m} , \varphi ) \Vol_{\mathbb{H}^2(-a^2)}.
\end{split}
\end{equation}
The fact that $\Psi$ satisfies \eqref{suprise} allows us to use  \cite[Lemma 4.2]{CC13} to deduce that

\begin{equation}\label{useMaster1}
\begin{split}
- \int_{\Omega (R_0 , \widetilde{R} )} g (\nabla_{\Psi} w_{R_m} , \varphi ) \Vol_{\mathbb{H}^2(-a^2)}
& = - \int_{\mathbb{H}^2(-a^2)}  g (\nabla_{\Psi} w_{R_m} , \varphi ) \Vol_{\mathbb{H}^2(-a^2)} \\
& = \int_{\mathbb{H}^2(-a^2)} g (w_{R_m}    , \nabla_{\Psi} \varphi  )  \Vol_{\mathbb{H}^2(-a^2)} \\
& = \int_{\Omega (R_0 , \widetilde{R} )}  g (w_{R_m}    , \nabla_{\Psi} \varphi  )  \Vol_{\mathbb{H}^2(-a^2)},
\end{split}
\end{equation}
as well as that
\begin{equation}\label{Usemaster2}
-\int_{\Omega (R_0 , \widetilde{R} )}  g (\nabla_{w_{R_m}} w_{R_m} , \varphi ) \Vol_{\mathbb{H}^2(-a^2)}
= \int_{\Omega (R_0 , \widetilde{R} )} g ( w_{R_m} , \nabla_{w_{R_m}} \varphi   ) \Vol_{\mathbb{H}^2(-a^2)} .
\end{equation}

So, in light of \eqref{useMaster1} and \eqref{Usemaster2},   \eqref{RelationForPasstoLimit} can now be rephrased in the following equivalent, but more useful form.
\begin{equation}\label{UsefulForPasstoLimit}
\begin{split}
& (( w_{R_m}  , \varphi    ))_{H^1_0(\Omega (R_0))} \\
= & \Big < \Phi , \varphi \Big >_{\textbf{V}'  ( \Omega (R_0)  )\otimes \textbf{V} ( \Omega (R_0)  )}
- \int_{\Omega (R_0 ,\widetilde{R})} g ( \nabla_{w_{R_m}} \Psi  , \varphi  )\Vol_{\mathbb{H}^2(-a^2)} \\
&+ \int_{\Omega (R_0 , \widetilde{R} )}  g (w_{R_m}    , \nabla_{\Psi} \varphi  )  \Vol_{\mathbb{H}^2(-a^2)}
+ \int_{\Omega (R_0 , \widetilde{R} )} g ( w_{R_m} , \nabla_{w_{R_m}} \varphi   ) \Vol_{\mathbb{H}^2(-a^2)} .
\end{split}
\end{equation}

We note that we have the following compact inclusion \cite[Theorem 2.9]{Hebey}.
\begin{equation*}
H^1 \big ( \Omega (R_0 , \widetilde{R}) \big ) \subset \subset L^4\big ( \Omega (R_0 , \widetilde{R})  \big ).
\end{equation*}
So, the fact that $\big \{ w_{R_m}\big \}_{m=1}^{\infty}$ is bounded in $H^1 \big ( \Omega (R_0 , \widetilde{R}) \big )$ implies that, up to taking another subsequence if necessary, we will have the following strong convergence.
\begin{equation}\label{strongConverge}
\lim_{m\rightarrow \infty } \big \| w_{R_m} - \widetilde{w} \big \|_{L^4 ( \Omega (R_0 , \widetilde{R})   ) } = 0 .
\end{equation}
Observe  
\begin{equation}\label{ineq1}
\begin{split}
& \Big | \int_{\Omega (R_0 , \widetilde{R} )} g ( \nabla_{w_{R_m}} \Psi, \varphi  ) - g ( \nabla_{\widetilde{w}} \Psi , \varphi )  \Vol_{\mathbb{H}^2(-a^2)}  \Big | \\
\leq & \big \| w_{R_m} - \widetilde{w} \big \|_{L^4 ( \Omega (R_0 , \widetilde{R})   ) }
\big \| \nabla \Psi \big \|_{L^2(\mathbb{H}^2(-a^2))} \big \| \varphi \big \|_{L^4 (\mathbb{H}^2(-a^2))}.
\end{split}
\end{equation}
Thus \eqref{strongConverge} allows us to pass to the limit on \eqref{ineq1} and deduce that
\begin{equation}\label{Done1}
\lim_{m\rightarrow \infty } \int_{\Omega (R_0 , \widetilde{R} )} g (  \nabla_{w_{R_m}} \Psi , \varphi    ) \Vol_{\mathbb{H}^2(-a^2)} = \int_{\Omega (R_0 , \widetilde{R} )} g (\nabla_{\widetilde{w}} \Psi ,\varphi ) \Vol_{\mathbb{H}^2(-a^2)}.
\end{equation}
Arguing in the same manner, \eqref{strongConverge} immediately leads to
\begin{equation}\label{Done2}
\lim_{m \rightarrow \infty } \int_{\Omega (R_0 , \widetilde{R} )}  g (w_{R_m}    , \nabla_{\Psi} \varphi  )  \Vol_{\mathbb{H}^2(-a^2)}
= \int_{\Omega (R_0 , \widetilde{R} )}  g (\widetilde{w} , \nabla_{\Psi} \varphi )
\Vol_{\mathbb{H}^2(-a^2)}.
\end{equation}

Next
\begin{equation}\label{INEQ2}
\begin{split}
& \Big | \int_{\Omega (R_0 , \widetilde{R} )} g ( w_{R_m} , \nabla_{w_{R_m}} \varphi   )
- g (\widetilde{w} , \nabla_{\widetilde{w}} \varphi ) \Vol_{\mathbb{H}^2(-a^2)} \Big | \\
\leq & \Big | \int_{\Omega (R_0 , \widetilde{R} )} g ( w_{R_m} -\widetilde{w} , \nabla_{w_{R_m}} \varphi   ) \Vol_{\mathbb{H}^2(-a^2)} \Big |
+ \Big | \int_{\Omega (R_0 , \widetilde{R} )} g (\widetilde{w} , \nabla_{(w_{R_m} - \widetilde{w})} \varphi )\Vol_{\mathbb{H}^2(-a^2)} \Big | \\
& \leq \| w_{R_m} - \widetilde{w} \big \|_{L^4 ( \Omega (R_0 , \widetilde{R})   ) }
\big \|\nabla \varphi \big \|_{L^2(\mathbb{H}^2(-a^2))}
\Big \{ \big \| w_{R_m}\big \|_{L^4( \Omega (R_0 , \widetilde{R} )  )} + \big \|\widetilde{w} \big \|_{L^4( \Omega (R_0 , \widetilde{R} )  )}\Big \}.
\end{split}
\end{equation}

So, again by \eqref{strongConverge} we have
\begin{equation}\label{Done3}
\lim_{m\rightarrow \infty } \int_{\Omega (R_0 , \widetilde{R} )} g ( w_{R_m} , \nabla_{w_{R_m}} \varphi   ) \Vol_{\mathbb{H}^2(-a^2)}
= \int_{\Omega (R_0 , \widetilde{R} )}  g (\widetilde{w} , \nabla_{\widetilde{w}} \varphi )         \Vol_{\mathbb{H}^2(-a^2)}.
\end{equation}
Using \eqref{Done1}, \eqref{Done2}, \eqref{Done3}, as well as \eqref{Weakconvergence}, we can now pass to the limit in \eqref{UsefulForPasstoLimit} to deduce 
\begin{equation}\label{Alt}
\begin{split}
& (( \widetilde{w}  , \varphi    ))_{H^1_0(\Omega (R_0))} \\
= & \Big < \Phi , \varphi \Big >_{\textbf{V}'  ( \Omega (R_0)  )\otimes \textbf{V} ( \Omega (R_0)  )}
- \int_{\Omega (R_0 ,\widetilde{R})} g ( \nabla_{\widetilde{w}} \Psi  , \varphi  )\Vol_{\mathbb{H}^2(-a^2)} \\
&+ \int_{\Omega (R_0 , \widetilde{R} )}  g (\widetilde{w}    , \nabla_{\Psi} \varphi  )  \Vol_{\mathbb{H}^2(-a^2)}
+ \int_{\Omega (R_0 , \widetilde{R} )} g ( \widetilde{w} , \nabla_{\widetilde{w}} \varphi   ) \Vol_{\mathbb{H}^2(-a^2)} .
\end{split}
\end{equation}
Final application of  \cite[Lemma 4.2]{CC13} gives
\begin{equation}\label{ID}
\begin{split}
\int_{\Omega (R_0 , \widetilde{R} )}  g (\widetilde{w}    , \nabla_{\Psi} \varphi  )  \Vol_{\mathbb{H}^2(-a^2)}
& = -\int_{\Omega (R_0)} g (\nabla_{\Psi} \widetilde{w}    ,  \varphi  )\Vol_{\mathbb{H}^2(-a^2)} ,\\
\int_{\Omega (R_0 , \widetilde{R} )} g ( \widetilde{w} , \nabla_{\widetilde{w}} \varphi   ) \Vol_{\mathbb{H}^2(-a^2)}
& = - \int_{\Omega (R_0)} g ( \nabla_{\widetilde{w}} \widetilde{w}   , \varphi      )\Vol_{\mathbb{H}^2(-a^2)}.
\end{split}
\end{equation}
The two identities in \eqref{ID} allow us to convert \eqref{Alt} to \eqref{WEAKReally}.
Since the test $1$-form $\varphi \in \Lambda_{c, \sigma}^1 \big ( \Omega (R_0) \big ) $ was arbitrary, we conclude that $\widetilde{w}\in \textbf{V} \big ( \Omega (R_0) \big )$ is a weak solution to the system \eqref{TureNSSecond} as needed.

\subsection{Recovering the pressure from the weak formulation \eqref{WEAKReally}.}\label{EASYsubsection}

Let $\widetilde{w} \in \textbf{V}(\Omega (R_0))$  be the element which we constructed in Section \ref{approx}. We then consider a linear operator $\widetilde{L} : H^1_0(\Omega (R_0)) \rightarrow \mathbb{R}$ as follows.

\begin{equation}\label{OperatorWEAKReally}
\begin{split}
\Big <\widetilde{L} ,\varphi \Big > = & \int_{\Omega (R_0)} g (\dd \widetilde{w} , \dd \varphi ) \Vol_{\mathbb{H}^2(-a^2)} + 2a^2 \int_{\Omega (R_0)}
g(\widetilde{w}, \varphi ) \Vol_{\mathbb{H}^2(-a^2)} \\
&-  \Big < \Phi , \varphi \Big >_{H^{-1} ( \Omega (R_0)  )\otimes H^1_0 ( \Omega (R_0)  )}
+ \int_{\Omega (R_0)} g ( \nabla_{\widetilde{w}} \Psi  , \varphi  )\Vol_{\mathbb{H}^2(-a^2)} \\
& + \int_{\Omega (R_0)} g (\nabla_{\Psi} \widetilde{w} , \varphi ) \Vol_{\mathbb{H}^2(-a^2)}
+\int_{\Omega (R_0)}  g (\nabla_{\widetilde{w}} \widetilde{w} , \varphi ) \Vol_{\mathbb{H}^2(-a^2)}.
\end{split}
\end{equation}
By applying \eqref{Weired1} to $\nabla \Psi$ and \eqref{reallyEASY} to both $\widetilde{w}$ and $\varphi$, we deduce  
\begin{equation}\label{REALLYTrivial1}
\begin{split}
\Big | \int_{\Omega (R_0)} g ( \nabla_{\widetilde{w}} \Psi , \varphi    ) \Vol_{\mathbb{H}^2(-a^2)}  \Big |
\leq C(a,R_0) \big \| \nabla \widetilde{w} \big \|_{L^2(\Omega (R_0) )} \big \| \dd F \big \|_{L^2(\mathbb{H}^2(-a^2))}
\big \| \nabla \varphi \big \|_{L^2(\Omega (R_0))}.
\end{split}
\end{equation}
Similarly,  \eqref{weired2}  and \eqref{reallyEASY} give
\begin{equation}\label{REALLYTrivial2}
\begin{split}
\Big | \int_{\Omega (R_0)} g ( \nabla_{\Psi} \widetilde{w} , \varphi    ) \Vol_{\mathbb{H}^2(-a^2)}  \Big |
\leq  C(a,R_0) \big \| \dd F \big \|_{L^2(\mathbb{H}^2(-a^2))} \big \| \nabla \widetilde{w} \big \|_{L^2(\Omega (R_0))} \big \| \nabla \varphi \big \|_{L^2(\Omega (R_0))},
\end{split}
\end{equation}

and
\begin{equation}\label{REALLYTrivial3}
\begin{split}
\Big | \int_{\Omega (R_0)} g ( \nabla_{\widetilde{w}} \widetilde{w} , \varphi    ) \Vol_{\mathbb{H}^2(-a^2)}  \Big |
& = \Big | \int_{\Omega (R_0)} g (  \widetilde{w} , \nabla_{\widetilde{w}} \varphi    ) \Vol_{\mathbb{H}^2(-a^2)}  \Big | \\
& \leq \big \| \widetilde{w} \big \|_{L^4(\Omega (R_0))}^2 \big \| \nabla \varphi \big \|_{L^2(\Omega (R_0))} \\
& \leq C_a \big \| \nabla \widetilde{w} \big \|_{L^2(\Omega (R_0))}^2 \big \| \nabla \varphi \big \|_{L^2(\Omega (R_0))}.
\end{split}
\end{equation}

Note that the very same sequence of logical steps as demonstrated in \eqref{TEDIOUS1}, \eqref{EXTRA1}, \eqref{Extra2}, \eqref{Extra3}   gives the following   analog of \eqref{Phidualest1}.

\begin{equation}\label{PhiDualest31}
\big \| \Phi \big \|_{H^{-1} (\Omega (R_0) )} \leq C(a,R_0) \Big \{  \big \| \dd F \big \|_{L^2(\mathbb{H}^2(-a^2))}
  +   \big \| \dd F \big \|_{L^2(\mathbb{H}^2(-a^2))}^2     \Big \} .
\end{equation}

\eqref{PhiDualest31} justifies the use of the notation $  \Big < \Phi , \varphi \Big >_{H^{-1} ( \Omega (R_0)  )\otimes H^1_0 ( \Omega (R_0)  )}$ in the definition \eqref{OperatorWEAKReally} for the operator $\widetilde{L}$.
In any case, by combining \eqref{REALLYTrivial1}-\eqref{PhiDualest31}, we deduce  $\widetilde{L} \in H^{-1} (\Omega (R_0))$. Next, the fact that $\widetilde{w}$ satisfies the weak formulation \eqref{WEAKReally} simply tells us that the restriction of the linear operator $\widetilde{L}$ on $\textbf{V}(\Omega (R_0))$ vanishes identically. That is, we now have
\begin{equation}\label{FinalCondition}
\begin{split}
& \widetilde{L}  \in H^{-1} \big ( \Omega (R_0) \big ) , \\
& \widetilde{L} \big |_{\textbf{V}(\Omega (R_0) )}  = 0 .
\end{split}
\end{equation}

\eqref{FinalCondition} allows us to invoke Lemma \ref{RecoverPressure} to deduce that there exists some $P \in L^2_{loc}(\Omega (R_0))$ such that the pair $(\widetilde{w} , P)$ constitutes a solution to the system \eqref{TureNSSecond} on $\Omega (R_0)$. 

So summarizing up to now, by considering the element $v$ as given by   \eqref{FormoftheSol}, with $w$ to be given by Lemma \ref{KeyLEMMA}, it follows that $v \in H^1_0 (\Omega (R_0))$, and that the pair $(v, P)$ constitutes a solution to the system \eqref{TrueNavierStokeseq} on $\Omega (R_0)$. Moreover, Lemma \ref{JudgementTWO} immediately tells us that such an element $v$   must be  nontrivial.

\subsection{The proof of Lemma \ref{ReallyStandardresult} via Leray-Schauder fixed point principle.}\label{RoutineSubsection}

Again, this is based on \cite{Seregin}.  To begin, let ${C}(a,R_0)$ be the same absolute constant as specified in Lemma \ref{ImportantLemma}, and let $\dd F$ be given and satisfy  \eqref{ConditionondFTWO}. 

Fix $R > 5 R_0$. For any arbitrary  $\theta \in \textbf{V} (\Omega (R_0 , R ))$, we consider the linear operator $\textbf{L}(\theta) : \textbf{V} (\Omega (R_0 , R)) \rightarrow \mathbb{R}$ defined by
\begin{equation}\label{operatorLeraySch}
\begin{split}
\Big < \textbf{L}(\theta ) ,\varphi \Big > = & \Big < \Phi , \varphi \Big >_{\textbf{V}'  ( \Omega (R_0 , R )  )\otimes \textbf{V} ( \Omega (R_0 , R )  )}
- \int_{\Omega (R_0 , R)} g ( \nabla_{\theta} \Psi  , \varphi  )\Vol_{\mathbb{H}^2(-a^2)} \\
& - \int_{\Omega (R_0 , R )} g (\nabla_{\Psi} \theta , \varphi ) \Vol_{\mathbb{H}^2(-a^2)}
-\int_{\Omega (R_0 , R )}  g (\nabla_{\theta} \theta , \varphi ) \Vol_{\mathbb{H}^2(-a^2)}.
\end{split}
\end{equation}

Now, by merely replacing $\widetilde{w}$ by $\theta$ and $\Omega (R_0)$ by $\Omega (R_0 , R)$ in the derivations of estimates \eqref{REALLYTrivial1}, \eqref{REALLYTrivial2}, and \eqref{REALLYTrivial3}, we easily obtain the following three estimates 
\begin{equation}\label{Well1}
\begin{split}
&\Big | \int_{\Omega (R_0, R )} g ( \nabla_{\theta} \Psi , \varphi    ) \Vol_{\mathbb{H}^2(-a^2)}  \Big | \\
 \leq & C(a,R_0) \big \| \nabla \theta \big \|_{L^2(\Omega (R_0, R ) )} \big \| \dd F \big \|_{L^2(\mathbb{H}^2(-a^2))}
\big \| \nabla \varphi \big \|_{L^2(\Omega (R_0 ,R ))} ,
\end{split}
\end{equation}

\begin{equation}\label{Well2}
\begin{split}
& \Big | \int_{\Omega (R_0, R )} g ( \nabla_{\Psi} \theta , \varphi    ) \Vol_{\mathbb{H}^2(-a^2)}  \Big | \\
\leq & C(a,R_0) \big \| \dd F \big \|_{L^2(\mathbb{H}^2(-a^2))} \big \| \nabla \theta \big \|_{L^2(\Omega (R_0 , R))} \big \| \nabla \varphi \big \|_{L^2(\Omega (R_0 , R))} ,
\end{split}
\end{equation}

and

\begin{equation}\label{Well3}
\begin{split}
 \Big | \int_{\Omega (R_0, R)} g ( \nabla_{\theta} \theta , \varphi    ) \Vol_{\mathbb{H}^2(-a^2)}  \Big |
 \leq C_a \big \| \nabla \theta \big \|_{L^2(\Omega (R_0 , R))}^2 \big \| \nabla \varphi \big \|_{L^2(\Omega (R_0 , R))}.
\end{split}
\end{equation}
Through combining \eqref{Well1}-\eqref{Well3}, and \eqref{PhiDualest31}, we deduce  
\begin{equation}\label{Notuseful}
\begin{split}
\big \|\textbf{L}(\theta ) \big \|_{\textbf{V}'(\Omega (R_0 ,R  ))}
\leq  C(a,R_0) & \Big \{ \big \|\dd F \big \|_{L^2(\mathbb{H}^2(-a^2))} + \big \|\dd F \big \|_{L^2(\mathbb{H}^2(-a^2))}^2  \\
& + \big \| \nabla \theta \big \|_{L^2(\Omega (R_0 , R ))}\big \|\dd F \big \|_{L^2(\mathbb{H}^2(-a^2))}     + \big \| \nabla \theta \big \|_{L^2(\Omega (R_0 , R ))}^2        \Big \}.
\end{split}
\end{equation}

Now, we consider the following inner product  on $\textbf{V} (\Omega (R_0 , R))$.
\begin{equation*}
(( \varphi_1 , \varphi_2  ))_{\textbf{V}(\Omega (R_0 , R))}
= \int_{\Omega (R_0 , R )}  g(\dd \varphi_1 , \dd \varphi_2) + 2a^2 g (\varphi_1 , \varphi_2 )\Vol_{\mathbb{H}^2(-a^2)}.
\end{equation*}

The Riesz Representation Theorem gives us the natural isomorphism

\begin{equation*}
\textbf{T}_{Riesz} :\textbf{V}'(\Omega (R_0 ,R)) \rightarrow \textbf{V}(\Omega (R_0 , R)  ) ,
\end{equation*}
which is characterized by the following relation.

\begin{equation*}
\Big < \textbf{l} , \varphi \Big >_{\textbf{V}'(\Omega (R_0 ,R))\otimes \textbf{V}(\Omega (R_0 , R)  )}
= ((\textbf{T}_{Riesz} (\textbf{l}) , \varphi    ))_{\textbf{V}(\Omega (R_0 , R ))} ,
\end{equation*}
where $\textbf{l} \in \textbf{V}'(\Omega (R_0 ,R))$ and $\varphi \in \textbf{V}(\Omega (R_0 , R)  ) $ are arbitrary. \\

So, we can now consider the map $\textbf{B} : \textbf{V} (\Omega (R_0 , R ) ) \rightarrow \textbf{V} (\Omega (R_0 , R ))$ which is defined as follows.
\begin{equation}\label{B}
\textbf{B} (\theta ) =  \textbf{T}_{Riesz} ( \textbf{L} (\theta )   ) ,
\end{equation}
where $\theta \in \textbf{V} (\Omega (R_0 , R ))$ is arbitrary.

Just observe that proving the existence of an element $w_R \in \textbf{V} (\Omega (R_0 , R))$ satisfying the weak formulation \eqref{approxWEAK} is logically equivalent to proving the existence of a solution element $\theta \in \textbf{V}(\Omega (R_0 , R))$ to the following equation.
\begin{equation}\label{reformulation1}
\textbf{B} (\theta) = \theta .
\end{equation}

According to the Leray-Schauder fixed point theorem, one can deduce the existence of a solution $\theta \in \textbf{V}(\Omega (R_0 , R))$ to \eqref{reformulation1}, provided one can verify the following condtions
\begin{itemize}
\item  [I)] The operator $\textbf{B}$ is a compact operator from $\textbf{V} (\Omega (R_0 , R) )$ into itself.
\item [II)] There exists some constant $D_0 > 0$ such that for any $\lambda \in (0,1]$, and whenever we have an element $\theta_{\lambda} \in \textbf{V}(\Omega (R_0 ,R))$ which satisfies the relation
    \begin{equation}\label{perturbequation}
    \theta_{\lambda} = \lambda \textbf{B} (\theta_{\lambda}) ,
    \end{equation}
    it follows that we have $\| \nabla \theta_{\lambda } \|_{L^2(\Omega (R_0 , R))} \leq D_0$.
\end{itemize}

Through similar arguments as in Section \ref{approx} we can show $\textbf{B}$ is compact; we omit the details.  We show however the details for Condition II above.  Hence, suppose that $\theta_{\lambda} \in \textbf{V} (\Omega (R_0 , R))$ satisfies \eqref{perturbequation} for some given $\lambda \in (0,1]$. This means the same as saying that $\theta_{\lambda}$ satisfies the following relation for any test $1$-form $\varphi \in \textbf{V} (\Omega (R_0 , R))$.

\begin{equation}\label{PerturbWEAK}
\begin{split}
& \int_{\Omega (R_0 , R)} g (\dd \theta_{\lambda} , \dd \varphi ) \Vol_{\mathbb{H}^2(-a^2)} + 2a^2 \int_{\Omega (R_0 , R)}
g(\theta_{\lambda}, \varphi ) \Vol_{\mathbb{H}^2(-a^2)} \\
= &\lambda \Big < \Phi , \varphi \Big >_{\textbf{V}'  ( \Omega (R_0 , R )  )\otimes \textbf{V} ( \Omega (R_0 , R )  )}
- \lambda \int_{\Omega (R_0 , R)} g ( \nabla_{\theta_{\lambda}} \Psi  , \varphi  )\Vol_{\mathbb{H}^2(-a^2)} \\
& - \lambda \int_{\Omega (R_0 , R )} g (\nabla_{\Psi} \theta_{\lambda} , \varphi ) \Vol_{\mathbb{H}^2(-a^2)}
-\lambda \int_{\Omega (R_0 , R )}  g (\nabla_{\theta_{\lambda}} \theta_{\lambda} , \varphi ) \Vol_{\mathbb{H}^2(-a^2)}.
\end{split}
\end{equation}

So, by taking $\varphi$ to be $\theta_{\lambda}$ in \eqref{PerturbWEAK}, we can deduce 

\begin{equation}\label{substitutetheta}
\begin{split}
& \big \| \nabla \theta_{\lambda} \big \|_{L^2(\Omega ( R_0 , R ))}^2 + a^2 \big \| \theta_{\lambda} \big \|_{L^2(\Omega (R_0 , R))}^2 \\
= & \big \| \dd \theta_{\lambda} \big \|_{L^2(\Omega ( R_0 , R ))}^2 + 2a^2 \big \| \theta_{\lambda} \big \|_{L^2(\Omega (R_0 , R))}^2 \\
= & \lambda \Big < \Phi , \theta_{\lambda} \Big >_{\textbf{V}'  ( \Omega (R_0 , R )  )\otimes \textbf{V} ( \Omega (R_0 , R )  )}
- \lambda \int_{\Omega (R_0 , R)} g ( \nabla_{\theta_{\lambda}} \Psi  , \theta_{\lambda}  )\Vol_{\mathbb{H}^2(-a^2)} \\
 & - \lambda \int_{\Omega (R_0 , R )} g (\nabla_{\Psi} \theta_{\lambda} , \theta_{\lambda} ) \Vol_{\mathbb{H}^2(-a^2)}
- \lambda \int_{\Omega (R_0 , R )}  g (\nabla_{\theta_{\lambda}} \theta_{\lambda} , \theta_{\lambda} ) \Vol_{\mathbb{H}^2(-a^2)}.
\end{split}
\end{equation}
It is important to make the following observation:

\begin{itemize}
\item  The structure of \eqref{substitutetheta} is essentially identical to that of \eqref{substituteWR}. The exception is that there is just now an extra multiplicative factor $\lambda \in (0,1]$ on the right-hand side of \eqref{substitutetheta} (and that in \eqref{substitutetheta}, the term $w_R$ is now replaced by $\theta_{\lambda}$).
\end{itemize}

The above observation allows us to conclude that the sequence of  the estimates done for $w_R$ in Section \ref{subsection3.4} applies  to $\theta_{\lambda}$, and this allows us to conclude that $\theta_{\lambda}$ definitely satisfies the very same a priori estimate \eqref{GOODESTIMATE} as $w_R$ does. That is, we must have

\begin{equation}\label{ThetaGOODESTIMATE}
\big \| \nabla \theta_{\lambda} \big \|_{L^2(\Omega (R_0 ,R))}^2 \leq
\frac{\big ( {C}(a,R_0) \big )^2 \Big \{ \big \| \dd F \big \|_{L^2(\mathbb{H}^2(-a^2))} + \big \| \dd F \big \|_{L^2(\mathbb{H}^2(-a^2))}^2     \Big \}^2  }{\Big ( 1- 2{C}(a,R_0)  \|\dd F \|_{L^2(\mathbb{H}^2(-a^2))} \Big ) } .
\end{equation}

So, the a priori estimate \eqref{ThetaGOODESTIMATE} indicates that Condition II has also been verified. So, the Leray-Schauder fixed point theorem can be applied to the operator $\textbf{B}$ to deduce the existence of an element $w_R \in \textbf{V} (\Omega (R_0 ,R))$ which satisfies the relation $w_R =  \textbf{B} (w_R)$. This, however, means the same as saying that we now have an element $w_R$ which satisfies the weak formulation \eqref{approxWEAK}. So, the proof of Lemma \ref{ReallyStandardresult} is completed.

\subsection{The main result arrived through our discussions in this section.}

Here we just summarize the arguments in the previous sections, which can also be viewed as a detailed version of Theorem \ref{main2}.
\begin{theo}
Given $R_0 > 0$ and $a > 0$. Then, there exists an absolute constant ${C} (a, R_0) > 0$, which depends only on $R_0$ and $a$, such that for any harmonic $1$-form $\dd F \in L^2(\mathbb{H}^2(-a^2))$ which satisfies the following estimate
\begin{equation*}
\big \| \dd F \big \|_{L^2(\mathbb{H}^2(-a^2))}  < \frac{1}{2 {C}(a ,R_0 )} ,
\end{equation*}
there exists an element $\widetilde{w} \in \textbf{V}\big ( \Omega (R_0) \big ) $ and some $P \in L^2_{loc} (\Omega (R_0))$ such that the pair $(\widetilde{w},P )$  is a solution to \eqref{TureNSSecond} on $\Omega (R_0)$, where $w$ is the one as specified in Lemma \ref{KeyLEMMA}. In addition, such an element $\widetilde{w}$ satisfies the following a priori estimate.
\begin{equation*}
\big \| \nabla \widetilde{w} \big \|_{L^2(\Omega (R_0 ))}^2 \leq
\frac{\big ( {C}(a,R_0) \big )^2 \Big \{ \big \| \dd F \big \|_{L^2(\mathbb{H}^2(-a^2))} + \big \| \dd F \big \|_{L^2(\mathbb{H}^2(-a^2))}^2     \Big \}^2  }{\Big ( 1- 2  {C}(a,R_0)  \|\dd F \|_{L^2(\mathbb{H}^2(-a^2))} \Big ) } .
\end{equation*}
Moreover, if we consider the element $v$ which is defined by
\begin{equation*}
v = \big ( \eta_{R_0} - 1 \big ) \dd F + w + \widetilde{w} ,
\end{equation*}
it follows that $v$ is a nontrivial element in $H^1_0\big (\Omega (R_0) \big )$, and that the pair $(v , P)$ constitutes a solution to the stationary Navier-Stokes equation \eqref{TrueNavierStokeseq} on $\Omega (R_0)$.
\end{theo}
\appendix
\section{Appendix: About recovering the pressure term from the Stokes- or Navier-Stokes equations.   }
In this section, we establish a lemma which we use to recover the pressure term of either the Stationary Stokes equation or the Stationary Navier-Stokes equation from their respective weak formulations. In our discussion, we use the following  result from the Euclidean theory.

\begin{lemm}\cite[Proposition 6.1, Section 1.6 ]{Seregin}\label{pressureLemmaEuc}
Let $\Omega$  be a bounded domain with Lipschitz boundary in $\mathbb{R}^N$, with $N = 2,3$.
Let
\begin{equation}
H^1_0 (\Omega ) = \overline{\Lambda_{c}^1 \big ( \Omega  \big ) }^{\|\cdot \|_{H^1(\Omega )}}.
\end{equation}
 Consider a bounded linear functional $\textbf{l} :  H_0^1(\Omega ) \rightarrow \mathbb{R}$ such that 
  \begin{equation}\label{Vanishing}
 \Big < \textbf{l} , \varphi \Big >_{H^{-1}(\Omega) \otimes H_0^1(\Omega )} = 0 ,
 \end{equation}
 for any test vector field $\varphi \in H_0^1(\Omega )$ which satisfies $\Div_{\mathbb{R}^N} \varphi = 0$ on $\Omega$.\\
 Then, it follow that there exists a function $P \in L^2(\Omega )$ such that the relation
 \begin{equation}\label{conditionPressure}
 \Big < \textbf{l} , \varphi \Big >_{H^{-1}(\Omega) \otimes H_0^1(\Omega )} = - \int_{\Omega } P \Div_{\mathbb{R}^N} \varphi \Vol_{\mathbb{R}^N}
 \end{equation}
 holds for any $\varphi \in H_0^1(\Omega )$.
\end{lemm}

What follows is somewhat standard, but because of the different setting we include the details for completeness.  The pressure is recovered in a similar fashion also in \cite{Seregin}.

Since for any bounded domain $\Omega$ with Lipschitz boundary in $\mathbb{R}^N$,  $N = 2,3$,
\begin{equation*}
\big \{ \varphi \in H_0^1(\Omega ) : \Div_{\mathbb{R}^N} \varphi  = 0  \big \} = \overline{\Lambda_{c, \sigma }^1 \big ( \Omega  \big ) }^{\|\cdot \|_{H^1(\Omega )}},
\end{equation*}
to apply Lemma \ref{pressureLemmaEuc} one has to check the validity of relation \eqref{Vanishing} only for those test vector fields $\varphi$ in $\Lambda_{c,\sigma}^1(\Omega )$.\\

Recall  
\begin{equation*}
\begin{split}
\Omega (r_1) & = \big \{ x \in \mathbb{H}^2(-a^2) : \rho (x) > r_1 \big \} , \\
\Omega (r_1 , r_2) & = \big \{ x \in \mathbb{H}^2(-a^2) : r_1 < \rho (x) < r_2     \big \} .
\end{split}
\end{equation*}
In this section, we also use for any $0 < r_1 < r_2 < \infty$
\begin{equation*}
A( r_1 , r_2) = \Big \{ y \in D_0(1) :  \tanh \big ( \frac{a}{2} r_1 \big )< |y| < \tanh\big ( \frac{a}{2}r_2 \big ) \Big \} .
\end{equation*}
Then, by \eqref{easy_inc2} the following uniform estimate holds for any $1$-form
$\varphi = \varphi_1 \dd Y^1 + \varphi_2 \dd Y^2 \in \Lambda_{c }^1 (\Omega ( r_1 , r_2 ) )$
\begin{equation}\label{verytrivial}
\big \| \nabla \varphi \big \|_{L^2(\Omega (r_1, r_2))} \leq C_a \big \| \nabla_{\mathbb{R}^2} \varphi^{\sharp} \big \|_{L^2( A( r_1 ,r_2)   )} ,
\end{equation}
with $\varphi^{\sharp}$ to be given by
\begin{equation}\label{pullbackofVarphi}
\varphi^{\sharp} = \varphi_1 \circ Y^{-1} \dd y^1 + \varphi_2 \circ Y^{-1} \dd y^2 .
\end{equation}
 
Now, let $R_0 > 0$ to be fixed, and  

\begin{equation*}
H^1_0(\Omega (R_0)) = \overline{\Lambda_{c}^1 (\Omega (R_0))}^{\|\cdot \|_{H^1(\Omega (R_0))}} .
\end{equation*}

Next, consider the bounded linear functional $\textbf{L} : H^1_0(\Omega (R_0)) \rightarrow \mathbb{R}$  such that
\begin{equation}\label{hypvanish}
\Big < \textbf{L} , \varphi \Big >_{H^{-1} (\Omega (R_0)) \otimes H^1_0 (\Omega (R_0))} = 0,
\end{equation}
holds for any $\varphi \in \Lambda^1_{c, \sigma } (\Omega (R_0))$. 

Now, take any strictly increasing sequence $\{R_k \}_{k=1}^{\infty} \subset (R_0 , + \infty ) $ such that
\begin{equation*}
\lim_{m\rightarrow \infty} R_m = + \infty ,
\end{equation*}
and carry out an inductive argument as follows. For each  $m\in \mathbb{Z}^+$, look at the region $\Omega (R_0 , R_m)$, and define a linear map $\textbf{L}_m^{\sharp} : H^1_0 (A(R_0 , R_m)) \rightarrow \mathbb{R}$ by
\begin{equation}\label{quiteNatural1}
\big < \textbf{L}_m^{\sharp} , \varphi^{\sharp} \big > = \Big < \textbf{L} , \varphi \Big >_{H^{-1} (\Omega (R_0)) \otimes H^1_0 (\Omega (R_0))} ,
\end{equation}
where $\varphi$ and $\varphi^{\sharp}$ are related via relation \eqref{pullbackofVarphi}. Now, through applying estimate \eqref{verytrivial}, we deduce from \eqref{quiteNatural1} that
 
\begin{equation*}
\begin{split}
\Big | \big < \textbf{L}_m^{\sharp} , \varphi^{\sharp} \big > \Big | & = \Big | \Big < \textbf{L} , \varphi \Big >_{H^{-1} (\Omega (R_0)) \otimes H^1_0 (\Omega (R_0))} \Big | \\
& \leq \big \| \textbf{L} \big \|_{H^{-1} (\Omega (R_0))} \big \| \nabla \varphi \big \|_{L^2(\Omega (R_0 , R_m))} \\
& \leq C_a\big \| \textbf{L} \big \|_{H^{-1} (\Omega (R_0))} \big \|\nabla_{\mathbb{R}^2} \varphi^{\sharp} \big \|_{L^2(A(R_0 , R_m))}.
\end{split}
\end{equation*}
Hence
\begin{equation}\label{condit1}
\big \| \textbf{L}^\sharp_m \big \|_{H^{-1} (A(R_0 , R_m))} \leq C_a\big \| \textbf{L} \big \|_{H^{-1} (\Omega (R_0))}.
\end{equation}

Since by \eqref{hypdiv}, $\dd^* \varphi = 0$ holds on $\Omega (R_0 , R_m)$ if and only if $\Div_{\mathbb{R}^2} \varphi^{\sharp} = 0$ holds on $A(R_0 , R_m)$, it follows from \eqref{quiteNatural1} that we must have
\begin{equation}\label{condit2}
\textbf{L}_m^{\sharp} \big |_{\textbf{V}( A(R_0 , R_m)  )} = 0 ,
\end{equation}
where  
\begin{equation*}
\textbf{V} \big ( A(R_0 , R_m) \big ) = \overline{\Lambda_{c, \sigma}^1\big ( A(R_0 , R_m)\big ) }^{\|\cdot \|_{H^1(A(R_0 , R_m))}}.
\end{equation*}

     Property \eqref{condit1} and property \eqref{condit2} together ensure that Lemma \ref{pressureLemmaEuc} can be directly applied to each $\textbf{L}_{m}^{\sharp}$.

So, we apply Lemma \ref{pressureLemmaEuc} to $\textbf{L}_1^{\sharp}$ to deduce that there exists some function
$P_1^{\sharp} \in L^2(A(R_0 , R_1))$ such that the following relation holds in the weak sense.
\begin{equation*}
\textbf{L}_1^{\sharp} = \nabla_{\mathbb{R}^2} P_1^{\sharp} ,
\end{equation*}
from which we deduce that the following relation holds for any $\varphi \in \Lambda_{c}^1 (\Omega (R_0 , R_1))$
\begin{equation}\label{same1}
\begin{split}
\Big < \textbf{L} , \varphi \Big >_{H^{-1} (\Omega (R_0)) \otimes H^1_0 (\Omega (R_0))}
& = \Big < L_1^{\sharp} , \varphi^{\sharp} \Big >_{H^{-1} (A(R_0 , R_1)) \otimes H^1_0 (A (R_0 , R_1))} \\
& = - \int_{A(R_0 , R_1 )} P_1^{\sharp} \Div_{\mathbb{R}^2} \varphi^{\sharp} \Vol_{\mathbb{R}^2} \\
& = \int_{\Omega (R_0 , R_1)} P_1 \dd^* \varphi \Vol_{\mathbb{H}^2(-a^2)} ,
\end{split}
\end{equation}
where the function $P_1 \in L^2(\Omega (R_0 , R_1) )$ is exactly given by
\begin{equation*}
P_1 = P_1^{\sharp} \circ Y .
\end{equation*}
So, by repeating the same argument, we deduce that for each $m \geq 2$, there exists some function $P_m \in L^2(\Omega (R_0 , R_m))$ such that the following relation holds for all test $1$-forms $\varphi \in \Lambda_{c}^1 (\Omega (R_0 , R_m))$.
\begin{equation}\label{same2}
\begin{split}
\Big < \textbf{L} , \varphi \Big >_{H^{-1} (\Omega (R_0)) \otimes H^1_0 (\Omega (R_0))}
 = \int_{\Omega (R_0 , R_m)} P_m \dd^* \varphi \Vol_{\mathbb{H}^2(-a^2)} .
\end{split}
\end{equation}
\eqref{same1} together with \eqref{same2} tells us 
\begin{equation*}
\dd \big ( P_m  - P_1 \big ) = 0
\end{equation*}
holds on $\Omega (R_0 , R_1)$, and hence there exists some constant $C_{1,m} \in \mathbb{R}$ for which  
\begin{equation*}
P_m \big |_{\Omega (R_0 , R_1)} = P_1 + C_{1,m} .
\end{equation*}
So, we can define the new function $\widetilde{P}_m \in L^2(\Omega (R_0 , R_m))$ by

\begin{equation*}
\widetilde{P}_m  = P_m - C_{1,m} .
\end{equation*}

The important point to make here is that $\widetilde{P}_m$ satisfies $\widetilde{P}_m \big |_{ \Omega (R_0 , R_1)} = P_1$, while the original $P_m$ may not.  The function $\widetilde{P}_m$ now satisfies
\begin{equation}
\Big < \textbf{L} , \varphi \Big >_{H^{-1} (\Omega (R_0)) \otimes H^1_0 (\Omega (R_0))} = \int_{\Omega (R_0 , R_m)}
 \widetilde{P}_m \dd^* \varphi \Vol_{\mathbb{H}^2(-a^2)},
\end{equation}
for any $\varphi \in \Lambda_c^1(\Omega (R_0 , R_m))$.
Finally, take any two integers $m,n$ satisfying $1 < m < n$. We then have to check that  
\begin{equation}\label{gluingidentity}
\widetilde{P}_n \big |_{\Omega (R_0 , R_m)} = \widetilde{P}_m .
\end{equation}
Similarly as before, since we know that $\dd (\widetilde{P}_n - \widetilde{P}_m) = 0$ holds on $\Omega (R_0 , R_m)$, the following identity holds for some constant $C_{m,n} \in \mathbb{R}$.
\begin{equation*}
\widetilde{P}_n \big |_{\Omega (R_0 , R_m)} = \widetilde{P}_m + C_{m,n} ,
\end{equation*}
form which it follows that
\begin{equation*}
\begin{split}
C_{m,n} & = \widetilde{P}_n \big |_{\Omega (R_0, R_1)} - \widetilde{P}_m \big |_{\Omega (R_0, R_1)} \\
&= \big \{ \widetilde{P}_n \big |_{\Omega (R_0, R_1)} -P_1 \} - \big \{\widetilde{P}_m \big |_{\Omega (R_0, R_1)}    - P_1 \big \} \\
& = 0.
\end{split}
\end{equation*}
This shows that \eqref{gluingidentity} definitely holds for any $1 < m < n$, and this allows us to construct a globally defined function $P \in L^2_{loc}(\Omega (R_0))$ in accordance to the following rule.

\begin{equation*}
P \big |_{\Omega (R_0 , R_m)} = \widetilde{P}_m ,
\end{equation*}
where $m \geq 2$. As a result, such a $P \in L^2_{loc}(\Omega (R_0))$ satisfies the following relation for each test $1$-form
$\varphi \in \Lambda_c^1 (\Omega (R_0))$.
\begin{equation*}
\Big < \textbf{L} , \varphi \Big >_{H^{-1}(\Omega (R_0)) \otimes H^1_0 (\Omega (R_0))} =  \int_{\Omega (R_0)} P \dd^* \varphi \Vol_{\mathbb{H}^2(-a^2)} .
\end{equation*}
In other words, the following relation holds on $\Omega (R_0)$ in the distributional sense.

\begin{equation*}
\textbf{L} = \dd P .
\end{equation*}

Our discussion in this section leads to the following lemma.

\begin{lemm}\label{RecoverPressure}
Let $R_0 > 0$, and
\begin{equation*}
H^1_0 (\Omega (R_0)) = \overline{\Lambda_c^1 (\Omega (R_0)  ) }^{\|\cdot \|_{H^1(\Omega (R_0))}}.
\end{equation*}
Consider a bounded linear functional $\textbf{L} : H^1_0 (\Omega (R_0)) \rightarrow \mathbb{R}$, which satisfies 
\begin{equation*}
\Big < \textbf{L} , \varphi \Big >_{H^{-1}(\Omega (R_0)) \otimes H^1_0 (\Omega (R_0))} = 0,
\end{equation*}
for all test $1$-forms $\varphi \in \textbf{V} (\Omega (R_0))$ .

Then, there exists a function $P \in L^2_{loc} (\Omega (R_0))$ such that the following relation holds for any
test $1$-form $\varphi \in \Lambda_c^1 (\Omega (R_0) )$.
\begin{equation*}
\Big < \textbf{L} , \varphi \Big >_{H^{-1}(\Omega (R_0)) \otimes H^1_0 (\Omega (R_0))} =  \int_{\Omega (R_0)} P \dd^* \varphi \Vol_{\mathbb{H}^2(-a^2)} .
\end{equation*}
In other words, the statement $\textbf{L} = \dd P$ holds on $\Omega (R_0)$ in the sense of distributions.
\end{lemm}

\bibliography{ref}
\bibliographystyle{plain}

\end{document}